\documentclass[11pt]{article}
\usepackage{arxiv}

\usepackage[utf8]{inputenc} % allow utf-8 input
\usepackage[T1]{fontenc}    % use 8-bit T1 fonts
\usepackage{hyperref}       % hyperlinks
\usepackage{xcolor}
\hypersetup{
    colorlinks=true,
    linkcolor=blue,
    filecolor=magenta,      
    urlcolor=cyan,
    }
\usepackage{url}            % simple URL typesetting
\usepackage{booktabs}       % professional-quality tables
\usepackage{amsfonts}       % blackboard math symbols
\usepackage{nicefrac}       % compact symbols for 1/2, etc.
\usepackage{microtype}      % microtypography
\usepackage{lipsum}		% Can be removed after putting your text content
\usepackage{graphicx}
\usepackage{doi}
\newenvironment{customproof}[1]
{\renewcommand\proofname{Proof of Theorem~\ref{#1}}\begin{proof}}
{\end{proof}}
\usepackage{amsmath}
\usepackage{bm}

\DeclareMathOperator{\diag}{diag}
\usepackage{enumerate}
\usepackage{enumitem}

\usepackage{amsthm}
\newtheorem{theorem}{Theorem}[section]
\newtheorem{lemma}[theorem]{Lemma}
\newtheorem{defn}[theorem]{Definition}
\usepackage{nicematrix}

\theoremstyle{remark}

\title{Modified Patankar Linear Multistep methods for production-destruction systems}

\author{ \href{https://orcid.org/0000-0003-1412-8702}{\includegraphics[scale=0.06]{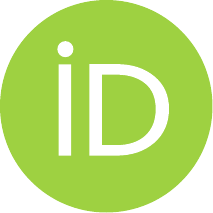}}\hspace{1mm}Giuseppe~Izzo\thanks{Members of the INdAM Research group 
GNCS} \\
	Department of Mathematics and Applications\\
	University of Naples ``Federico II''\\
	Via Cintia, I - 80126 Naples, Italy. \\
	\texttt{giuseppe.izzo@unina.it} \\
	\And
	\href{https://orcid.org/0000-0003-4545-6266}{\includegraphics[scale=0.06]{orcid.pdf}}\hspace{1mm}Eleonora~Messina$^*$ \\
	Department of Mathematics and Applications\\
	University of Naples ``Federico II''\\
	Via Cintia, I - 80126 Naples, Italy. \\
	\texttt{eleonora.messina@unina.it} \\
	\And
	\href{https://orcid.org/0000-0002-1869-945X}{\includegraphics[scale=0.06]{orcid.pdf}}\hspace{1mm}Mario~Pezzella$^*$ \\
	Institute for Applied Mathematics {\small``Mauro Picone''}\\
	C.N.R. National Research Council of Italy,\\
	Via P. Castellino, 111 - 80131 Naples - Italy. \\
	\texttt{mario.pezzella@cnr.it} \\
 \And
	\href{https://orcid.org/0000-0003-2746-7535}{\includegraphics[scale=0.06]{orcid.pdf}}\hspace{1mm}Antonia~Vecchio$^*$ \\
	Institute for Applied Mathematics {\small``Mauro Picone''}\\
	C.N.R. National Research Council of Italy,\\
	Via P. Castellino, 111 - 80131 Naples - Italy. \\
	\texttt{antonia.vecchio@cnr.it} \\
}

\hypersetup{
pdftitle={A template for the arxiv style},
pdfsubject={q-bio.NC, q-bio.QM},
pdfauthor={Giuseppe~Izzo, Eleonora~Messina, Mario~Pezzella, Antonia~Vecchio},
pdfkeywords={Patankar-type schemes, Positivity-preserving, High order, Conservativity, Linear multistep methods},
}

\begin{document}
\maketitle
\begin{abstract}
	Modified Patankar schemes are linearly implicit time integration methods designed to be unconditionally positive and conservative. In the present work we extend the Patankar-type approach to linear multistep methods and prove that the resulting discretizations retain, with no restrictions on the step size, the positivity of the solution and the linear invariant of the continuous-time system. Moreover, we provide results on arbitrarily high order of convergence and we introduce an embedding technique for the Patankar weights denominators to achieve it.
\end{abstract}

% keywords can be removed

\section{Introduction}
We address Production-Destruction Systems (PDS) of ordinary differential equations of the form 
\begin{equation}\label{eq:PDS}
	y_i^\prime(t)=P_i(\bm{y}(t))-D_i(\bm{y}(t)), \qquad y_i(0)=y_i^0, \qquad i=1,\dots, N,
\end{equation}
where $\bm{y}(t)=(y_1(t),\dots,y_N(t))^\mathsf{T}\in \mathbb{R}^N$ is the vector of constituents at time $t\geq 0.$ The right-hand side of \eqref{eq:PDS} incorporates the non-negative terms $P_i(\bm{y}(t))$ and $D_i(\bm{y}(t))$ defined as follows
\begin{equation}\label{eq:PDS_terms}
	P_i(\bm{y})=\sum_{j=1}^{N}p_{ij}(\bm{y}) \qquad \mbox{and} \qquad D_i(\bm{y})=\sum_{j=1}^{N}d_{ij}(\bm{y}), \qquad i=1,\dots, N,
\end{equation} 
where the functions $p_{ij}(\bm{y})\geq 0$ and $d_{ij}(\bm{y})\geq 0$ represent, respectively, the rates of production and destruction processes for each constituent and component. Here, we introduce the matrices $P(\bm{y})=\{p_{ij}(\bm{y})\}\in \mathbb{R}^{N\times N},$ $D(\bm{y})=\{d_{ij}(\bm{y})\}\in \mathbb{R}^{N\times N}$ and equivalently reformulate the PDS \eqref{eq:PDS} as
\begin{equation}\label{eq:sistema_compatto}
    \bm{y}^\prime(t)=(P(\bm{y}(t))-D(\bm{y}(t)))\bm{e}, \qquad\qquad \bm{y}(0)=\bm{y}^0,
\end{equation}
where here, and in the following sections, $\bm{e}=(1,\dots,1)^\mathsf{T}$ represents a vector of all ones with the appropriate number of components.

The mathematical modeling of various real life phenomena leads to differential systems of the form \eqref{eq:PDS}. PDS find relevant applications into chemical \cite{Higham2008,Kou2005} and biogeochemical \cite{BURCHARD2006,Semeniuk} processes, as well as into epidemic \cite{CAMPOS2021751,physics3020028}, ecosystem  \cite{fasham1990nitrogen,HENSE20102330,Warns} and astrophysical \cite{refId0} models  (see also \cite[Table 1]{anderHeiden1982} for a comprehensive list of applications). 

Our investigation is restricted to positive and fully conservative production-destruction systems, for which we assume that 
\begin{equation}\label{eq: pos_cont}
	y_i^0>0 \qquad \mbox{implies} \qquad y_i(t)>0,  \qquad \forall t> 0, \qquad i=1,\dots, N
\end{equation}
and that for each component-wise positive vector $\bm{z}\in \mathbb{R}^N,$
\begin{equation}\label{cond:b}
	p_{ij}(\bm{z})=d_{ji}(\bm{z}) \qquad \mbox{and} \qquad p_{ii}(\bm{z})=d_{ii}(\bm{z})=0, \qquad 1\leq i,j\leq N.
\end{equation}
The positivity condition \eqref{eq: pos_cont} is guaranteed as long as the initial value problem \eqref{eq:sistema_compatto} with positive initial value has a unique solution and the matrix $D(\bm{y})$ vanishes as $\bm{y}\to \bm{0}=(0,\dots,0)^\mathsf{T}\in \mathbb{R}^N$ (further discussions on this topic can be found in \cite{MPEuler,FORMAGGIA2011,TORLO2022}). General theoretical results about the existence, uniqueness and positivity of the solution of a production-destruction system have been outlined in \cite[Theorem 3.3]{FORMAGGIA2011} and \cite[Theorem 1.2]{TORLO2022}.  

The assumption \eqref{cond:b} leads to the linear invariant $\bm{e}^\mathsf{T}\bm{y}(t)$ and to the following conservation law 
\begin{equation}\label{eq:int_primo}
	\bm{e}^\mathsf{T}(\bm{y}(t)-\bm{y}^0)=0,\qquad \qquad \qquad \qquad\forall t \geq 0.
\end{equation}
As a matter of fact, for a fully conservative production-destruction system, the equality
\begin{equation*}
	\bm{e}^\mathsf{T}(P(\bm{y}(t))-D(\bm{y}(t)))\bm{e}=0, \quad \quad \qquad\forall t\geq 0,
    % \sum_{i=1}^{N}\left(P_i(\bm{y}(t))-D_i(\bm{y}(t))\right)=0,
\end{equation*}
holds true and from \eqref{eq:sistema_compatto}, $\left(\bm{e}^\mathsf{T}\bm{y}(t)\right)'=0.$ 

When applied to a positive and fully conservative PDS, a numerical method ought to satisfy a discrete counterpart of \eqref{eq: pos_cont} and \eqref{eq:int_primo}. However, the requirement for such properties in standard schemes usually yields a severe restriction on the stepsize. This motivates the interest in devising unconditionally positive methods which, provided the initial values are positive, produce positive numerical solutions independently of the steplength. Moreover, we will refer to a numerical method as unconditionally conservative if it retains the linear invariant of the system \eqref{eq:sistema_compatto} whatever the stepsize.

The development of unconditionally positive and conservative numerical methods for production-destruction differential systems has been addressed in several scientific contributions (see, for instance, \cite{NSFD_PDS_1,NSFD_PDS_2}). Particularly effective in this context is the class of Modified Patankar methods, arising from the manipulations of explicit schemes via the \textit{Patankar-trick} \cite{patankar1980numerical}, with the aim of ensuring the desired preservation properties for the numerical solution.  The origin of this class traces back to \cite{MPEuler} where a modification of forward Euler and Heun’s methods resulted in a first and second order positivity-preserving and conservative scheme, respectively. More recently, the same approach has been extended to Runge-Kutta methods and a general definition of Modified Patankar Runge-Kutta (MPRK) schemes has been introduced. In \cite{Kopecz2018,Kopecz2018second,Kopecz2019}, MPRK schemes of second and third order were presented and analyzed. Strong-Stability-Preserving MPRK (SSPMPRK) methods were designed in \cite{Huang2019,Huang2019II} to solve convection equations with stiff source terms, starting from the Shu-Osher form \cite{Shu_Form} of Runge-Kutta discretizations. A comprehensive analysis through center manifold theory of the stability of the MPRK and SSPMPRK schemes was carried out in \cite{Izgin1,Izgin2,Izgin3} and \cite{IzginSSPMPRK}, respectively. % There, the non trivial steady state solutions of general systems of linear differential equations were regarded as fixed-points for time integration methods and sufficient conditions for the stability were provided. 

The modified Patankar approach was not solely limited to Runge-Kutta methods. A predictor-corrector modified Patankar scheme was presented in \cite{FORMAGGIA2011}. The proposed discretization relied on a MPRK method as a predictor and on a modified Patankar three-steps backward differentiation formula for the correction step. High order Modified Patankar Deferred Correction (MPDeC) schemes were outlined in \cite{Torlo2020,TORLO2022} as effective integrators to mitigate common issues of MPRK discretizations, including oscillations around steady-state solutions and accuracy loss in cases of vanishing initial states.
The geometrically conservative non-standard integrators proposed in \cite{MARTIRADONNA2020, GeCo_Stab}, specifically designed to preserve linear invariants and to ensure the positivity of more general biochemical systems, are effectively employed for addressing production-destruction systems. Two modified Patankar methods based on specific second and third order multistep schemes were presented in  \cite{zhu2022MPLM}.

The leading purpose of this paper is to broaden the benefits of Patankar-type modifications to linear multistep methods, providing a general numerical framework for achieving high order of convergence. The manuscript is organized as follows: in Section \ref{sec:MPLM}, the class of unconditionally positive and conservative Modified Patankar Linear Multistep (MPLM) methods is formulated. With Section \ref{sec:Convergence} a theoretical investigation of the approximation error is performed and general results on arbitrarily high order of convergence are provided. Furthermore, a recursive embedding technique for the efficient computation of the Patankar weight denominators is presented. Numerical experiments are reported in Section \ref{sec:numerical_Experiments}, while some remarks and future perspectives, in Section \ref{sec:Conclusions}, conclude the paper.   

\section{The Modified Patankar Linear Multistep scheme}\label{sec:MPLM}
\label{sec:numerical_Method}
Let $\bm{\alpha}=(\alpha_1,\dots,\alpha_k)^\mathsf{T}\in \mathbb{R}^k$ and $\bm{\beta}=(\beta_1,\dots,\beta_k)^\mathsf{T}\in \mathbb{R}^k$ be the coefficients of an explicit $k$-steps Linear Multistep (LM) method, with $k>1$ positive integer. Assume that such a method is convergent of order $p\geq 1,$ which implies that
\begin{equation}\label{eq:Pre_Consistency}
    \sum_{r=1}^{k}\alpha_r=1 \qquad \quad \mbox{and} \qquad \quad
    \sum_{r=1}^{k}\left(r^q\alpha_r-qr^{q-1}\beta_r\right)=0, \qquad  1\leq q\leq p.
\end{equation}
Consider $h>0,$ $t_n=nh$ for $n\geq 0$ and $\bm{y}^n=(y_1^n,\dots,y_N^n)^\mathsf{T}\approx \bm{y}(t_n),$ for $n\geq k$.  From now on, all the inequalities involving vectors are considered component-wise. In analogy with the scientific literature on Patankar methods, we provide the following definition.
\begin{defn}
    Given the LM coefficients $\bm{\alpha}\geq 0$ and $\bm{\beta}\geq 0,$ the $k$-steps scheme
    \begin{equation}\label{eq:MPLM}
	y_i^n=\sum_{r=1}^{k}\alpha_ry_i^{n-r}+h\sum_{r=1}^{k}\beta_r\sum_{j=1}^{N}\left(p_{ij}(\bm{y}^{n-r})\dfrac{y_j^n}{\sigma_j^n}-d_{ij}(\bm{y}^{n-r})\dfrac{y_i^n}{\sigma_i^n}\right), \quad 1\leq i \leq N, \; \ n\geq k,
\end{equation}
where $\bm{y}^{0},\dots,\bm{y}^{k-1}\in \mathbb{R}^N$ are given and $\sigma_i^n\in \mathbb{R},$ $i=1,\dots,N,$ is referred to as a Modified Patankar Linear Multistep (MPLM-$k$) method if 
\begin{itemize}
    \item $\sigma_i^n$ are unconditionally positive for each $i=1,\dots,N$ and $n\geq k;$
    \item $\sigma_i^n$ are independent of $y_i^n$ for each $i=1,\dots,N$ and $n\geq k.$
\end{itemize}
\end{defn}

The linearly implicit numerical method \eqref{eq:MPLM} is devised by applying to a LM discretization of \eqref{eq:PDS} a modified version of the source term linearization technique \cite{Burchard2005} originally proposed by Patankar in \cite[Section 7.2-2]{patankar1980numerical}, with the aim of designing unconditionally positive and conservative schemes.
More specifically, the terms
\begin{equation*}    
	\sigma_i^n=\sigma_i^n(\bm{y}^{n-1},\dots, \bm{y}^{n-k}), \qquad \quad n\geq k, \qquad \quad i=1,\dots,N, 
\end{equation*}
are referred to as Patankar-Weight Denominators (PWDs). %In compliance with the scientific literature on Patankar methods, we consider unconditionally positive PWDs that are not constant during the integration and do not directly depend on the numerical solution at the current time-step. 
Notice that the discrete equation \eqref{eq:MPLM} is general enough to encompass the linear methods presented in \cite[Section 2.2]{zhu2022MPLM}. As a matter of fact, in case of 
\begin{equation*}
    \begin{split}
    &k=3, \qquad\quad \bm{\alpha}=\left(\dfrac{3}{4},0,\dfrac{1}{4}\right)^\mathsf{T}, \qquad\qquad \bm{\beta}=\left(0,\dfrac{3}{2},0\right)^\mathsf{T}, \qquad\quad\quad \mbox{and}
    \\
    &k=4, \qquad\quad \bm{\alpha}=\left(\dfrac{16}{27},0,0,\dfrac{11}{27}\right)^\mathsf{T}, \qquad \bm{\beta}=\left(\dfrac{16}{9},0,0,\dfrac{4}{9}\right)^\mathsf{T},
    \end{split}
\end{equation*}
\eqref{eq:MPLM} coincides with the second and third order schemes of \cite{zhu2022MPLM}, respectively.

Throughout this paper we refer, when needed, to the following equivalent compact notation of \eqref{eq:MPLM} 
\begin{equation}\label{eq:Compact_Expression}
    \bm{y}^n=\sum_{r=1}^k \alpha_r\bm{y}^{n-r} +h\sum_{r=1}^k\beta_rM^{n-r}\ \bm{y}^n, \qquad \qquad n\geq k,
\end{equation}
where 
\begin{equation}\label{eq:forma_M}
    M^{n-r}=\left( P(\bm{y}^{n-r})-\diag(D(\bm{y}^{n-r})\bm{e}) \right) \diag(\bm{S}^n)\in \mathbb{R}^{N\times N}, \qquad r=1,\dots,k,
\end{equation}
with $\bm{S}^n=(1/\sigma_1^n,\dots,1/\sigma_N^n)^\mathsf{T}\in \mathbb{R}^{N},$ $n\geq k.$

As already pointed out, the Patankar-type transformation is adopted to ensure positivity for the approximation of the solution to \eqref{eq:PDS}. The following result provides some conditions on the coefficients of \eqref{eq:MPLM}, or equivalently of \eqref{eq:Compact_Expression}, which lead to a positive numerical solution regardless of the discretization step-size.

\begin{lemma}\label{lem:positivity}
	Let $\{\bm{y}^n\}_{n\geq k}$ be the approximation of the solution to \eqref{eq:sistema_compatto} computed by \eqref{eq:Compact_Expression}. Suppose that \eqref{eq: pos_cont} and \eqref{cond:b} hold true and assume that
	\begin{enumerate}[label=$\mbox{h}_{\arabic*})$]
		% \item\label{hp:1} $\bm{\alpha}\geq 0$ and $\bm{\beta}\geq 0;$
		\item\label{hp:3} the given starting values $\bm{y}^{0},\dots,\bm{y}^{k-1}$ are positive.
	\end{enumerate} 
	Then, for all $h >0$ and $n\geq0,$ $\bm{y}^n>0.$
\end{lemma}
\begin{proof}
    From \eqref{eq:Compact_Expression}, the numerical method reads $M \bm{y}^n=\sum_{r=1}^k \alpha_r\bm{y}^{n-r},$ $n\geq k,$ with $M=I-h \sum_{r=1}^{k} \beta_r M^{n-r}\in \mathbb{R}^{N\times N}.$
    From \eqref{eq:forma_M}, the entries of $M^{n-r}$ are
	\begin{equation*}
		M^{n-r}_{ii}=-\sum_{j=1}^N \frac{d_{ij}(\bm{y}^{n-r})}{\sigma_i^n}, \qquad \; M^{n-r}_{ij}=\frac{p_{ij}(\bm{y}^{n-r})}{\sigma_j^n}, \qquad \; 1\leq i,j\leq N, \qquad \; i\neq j, 
	\end{equation*}
    for $n\geq k$ and $r=1,\dots, k,$ which implies that $M^\mathsf{T}$ is a Z-matrix, since 
	\begin{equation*}
			M_{ii}=1+h \sum_{r=1}^{k}\beta_r \sum_{j=1}^N \frac{d_{ij}(\bm{y}^{n-r})}{\sigma_i}>0 \quad \mbox{and} \quad
			M_{ji}=-h\sum_{r=1}^{k}\beta_r \frac{p_{ji}(\bm{y}^{n-r})}{\sigma_i}<0.
	\end{equation*}
    Furthermore, from \eqref{cond:b}, $M^\mathsf{T}$ is strictly diagonally dominant since
    \begin{equation*}
			\sum_{j=1,j\neq i}^N \left|M_{ji}\right|=h\sum_{r=1}^{k}\beta_r \sum_{j=1}^N \frac{d_{ij}(\bm{y}^{n-r})}{\sigma_i}<M_{ii}, \qquad \qquad \quad \ 1\leq i,j\leq N, \qquad \; i\neq j. 
	\end{equation*}
	Therefore $M^\mathsf{T}$ is a M-matrix (see, for instance, \cite[Lemma 6.2]{axelsson1996iterative}). It follows that $M$ is a M-matrix as well, it is invertible and $M^{-1}$ has all non-negative entries, independently of $n\geq k.$ Consequently $\bm{y}^n=M^{-1} \sum_{r=1}^k \alpha_r\bm{y}^{n-r}$ and an inductive procedure, starting from \ref{hp:3}, yields the result.
\end{proof}
%\begin{remark}
%    Under the assumptions of Lemma \ref{lem:positivity}, the numerical method \eqref{eq:MPLM} is unconditionally positive and will always provide positive numerical solutions, whatever the original PDS \eqref{eq:PDS} is positive and \eqref{eq: pos_cont} holds or not.
%\end{remark}

An investigation of the discrete-time level preservation of the invariance property \eqref{cond:b} of a positive and fully conservative PDS is carried out with the following lemma.  

\begin{lemma}\label{lem:conservativity}
	Let $\{\bm{y}^n\}_{n\geq k}$ be the approximation of the solution to \eqref{eq:sistema_compatto} computed by \eqref{eq:Compact_Expression}. Suppose that \eqref{eq: pos_cont} and \eqref{cond:b} hold true and assume that 
	\begin{enumerate}[label=$\mbox{h}^\prime_{\arabic*})$]
		\item\label{hp':2} for the given starting values, $\bm{e}^\mathsf{T}\bm{y}^{0}=\ldots=\bm{e}^\mathsf{T}\bm{y}^{k-1}=\eta.$
	\end{enumerate} 
	Then, for each $h >0$ and $n\geq 0,$ $\bm{e}^\mathsf{T}\left( \bm{y}^{n}-\bm{y}^{0}\right)=0.$
\end{lemma}
\begin{proof}
	It suffices to show that $\bm{e}^\mathsf{T}\bm{y}^{n}=\eta$ for each $n\geq k.$ The conservativity property \eqref{cond:b} yields, independently of $1\leq r \leq k$ and $n\geq k,$ 
    \begin{equation*}
		\begin{split}
			\sum_{i,j=1}^N\left(p_{ij}(\bm{y}^{n-r})\dfrac{y_j^n}{\sigma_j}-d_{ij}(\bm{y}^{n-r})\dfrac{y_i^n}{\sigma_i}\right)
            =\sum_{i,j=1}^N\left(p_{ij}(\bm{y}^{n-r})\dfrac{y_j^n}{\sigma_j}-p_{ji}(\bm{y}^{n-r})\dfrac{y_i^n}{\sigma_i}\right)=0,
		\end{split}
	\end{equation*}
	then, from \eqref{eq:Compact_Expression}, $\bm{e}^\mathsf{T}\bm{y}^{n}=\sum_{r=1}^{k} \alpha_r\left(\bm{e}^\mathsf{T}\bm{y}^{n-r}\right).$ The result is therefore derived from this last relation and the first of \eqref{eq:Pre_Consistency} by inductive arguments. % For $n=k,$ \ref{hp':2} yields $\bm{e}^\mathsf{T}\bm{y}^{k}=\bm{e}_k^\mathsf{T} \bm{\alpha} \ \eta=\eta.$
\end{proof}

Lemmas \ref{lem:positivity} and \ref{lem:conservativity} establish conditions for the method \eqref{eq:Compact_Expression} to be unconditionally positive and conservative. As a matter of fact, if their hypotheses are fulfilled, the numerical solution computed by the MPLM-$k$ scheme retains, with no restrictions on the step-length $h,$ the properties \eqref{eq: pos_cont} and \eqref{eq:int_primo} of the continuous-time PDS. Therefore to attain these properties we assume, from now on, that the starting values satisfy both \ref{hp:3} and \ref{hp':2}.

\section{Error analysis and convergence}\label{sec:Convergence}
In this section we analyze the error arising from the approximation of the continuous-time solution to \eqref{eq:PDS} by the numerical methods of the class \eqref{eq:Compact_Expression} and investigate the conditions that the Patankar weight denominators have to satisfy for attaining high order of convergence. A recursive practical technique, based on the embedding of different MPLM schemes, is then presented to efficiently compute the PWDs. 

In order to investigate the consistency of the discretization \eqref{eq:Compact_Expression}, we address the corresponding local truncation error, here denoted by 
\begin{equation}\label{eq:Local_Error}
	\begin{split}
		&\delta_i(h; t_n)=y_i(t_n)-\sum_{r=1}^{k}\alpha_ry_i(t_{n-r}) \\
        &-h\sum_{r=1}^{k}\beta_r\sum_{j=1}^{N}\left(p_{ij}(\bm{y}(t_{n-r}))\dfrac{y_j(t_n)}{\sigma_j(\bm{y}(t_{n-1}),\dots,\bm{y}(t_{n-k}))}-d_{ij}(\bm{y}(t_{n-r}))\dfrac{y_i(t_n)}{\sigma_i(\bm{y}(t_{n-1}),\dots,\bm{y}(t_{n-k}))}\right),
        % &-h\sum_{r=1}^{k}\beta_r\sum_{j=1}^{N}\left(\dfrac{p_{ij}(\bm{y}(t_{n-r}))y_j(t_n)}{\sigma_j(\bm{y}(t_{n-1}),\dots,\bm{y}(t_{n-k}))}-\dfrac{d_{ij}(\bm{y}(t_{n-r}))y_i(t_n)}{\sigma_i(\bm{y}(t_{n-1}),\dots,\bm{y}(t_{n-k}))}\right),
	\end{split}
\end{equation}
for $i=1,\dots,N$ and $n\geq k.$ 

The following result provides a condition on the PWDs which ensures the order $p$ consistency of the MPLM methods.

\begin{theorem}\textbf{(Sufficient Condition)}\label{thm:Consis_order_p}
Assume that the given functions describing problem \eqref{eq:PDS} belong to $C^p(\Omega_0),$ with $p\geq 1$ and $\Omega_0=\{z\in \mathbb{R}^N : 0\leq z_i\leq \bm{e}^\mathsf{T}\bm{y}^0,\ i=1,\dots,N\}.$  If the Patankar weight denominators satisfy 
\begin{equation}\label{eq:cond_sigma}
        \sigma_i(\bm{y}(t_{n-1}),\dots,\bm{y}(t_{n-k})) =y_i(t_n)+\mathcal{O}(h^{p}), \qquad i=1,\dots, N, \qquad n\geq k,
\end{equation}
then the MPLM-$k$ method \eqref{eq:Compact_Expression} is consistent with \eqref{eq:PDS}, of order $p$.
\end{theorem}
\begin{proof} % Forse qui occorre puntualizzare ancora che le y(t) sono positive, in modo da escludere lo zero. 
    It suffices to prove that $\delta_i(h; t_n)=\mathcal{O}(h^{p+1})$ for each $i=1,\dots,N$ and $n\geq k.$ 
    From \eqref{eq:cond_sigma} it follows that
    \begin{equation*}
        \dfrac{y_i(t_n)}{\sigma_i(\bm{y}(t_{n-1}),\dots,\bm{y}(t_{n-k}))}=\dfrac{y_i(t_n)}{y_i(t_n)+\mathcal{O}(h^{p})}=1+\mathcal{O}(h^{p}), \qquad i=1,\dots, N  
    \end{equation*}
    and the local truncation error defined in \eqref{eq:Local_Error} reads, for $n\geq k,$
    \begin{equation*}
            \delta_i(h; t_n)=\delta_i^{LM}(h;t_n)+\mathcal{O}(h^{p+1}), \qquad \qquad i=1,\dots, N,   \qquad \mbox{where}
    \end{equation*}
    \begin{equation}\label{eq:delta_LM}
        \delta_i^{LM}(h;t_n)=y_i(t_n)-\sum_{r=1}^{k}\alpha_ry_i(t_{n-r})-h\sum_{r=1}^{k}\beta_r\sum_{j=1}^{N}\left(p_{ij}(\bm{y}(t_{n-r}))-d_{ij}(\bm{y}(t_{n-r}))\right), 
    \end{equation}
    is the local truncation error of the underlying LM method with coefficients $\bm{\alpha}$ and $\bm{\beta}.$ Finally, since from \eqref{eq:Pre_Consistency} $\delta_i^{LM}(h;t_n)=\mathcal{O}(h^{p+1}),$ we get the result.
\end{proof}

Theorem \ref{thm:Consis_order_p} guarantees the consistency of the MPLM methods when the PWDs are selected to fulfil \eqref{eq:cond_sigma}. The following result, whose proof is reported in \hyperref[sec:Conclusions]{Appendix A}, states that this condition represents the minimum and the least stringent requirement for attaining the order $p$ consistency, as it constitutes a necessary prerequisite.

\begin{theorem}\textbf{(Necessary Condition)}\label{Them_Con_Neces}
    Assume that the MPLM-$k$ method \eqref{eq:Compact_Expression} is consistent with \eqref{eq:PDS} of order $p\geq 1.$ Then the Patankar weight denominators satisfy \eqref{eq:cond_sigma}.
\end{theorem}
    
To investigate the convergence properties of the scheme \eqref{eq:Compact_Expression}, we first establish some preparatory results. In what follows, given $\Omega\subset \mathbb{R}^N,$ we denote by $\Omega^l=\Omega \times \dots \times \Omega,$ $l$ times.
    \begin{lemma}\label{lem:diff}
        Let $l\geq 1$ be a positive integer and $\Omega$ be a compact subset of $\mathbb{R}^{N}.$ Consider, for $1\leq i,j \leq N$ and $\bm{x}=(\bm{x}^1,\dots,\bm{x}^l)\in \Omega^l,$ the functions 
        \begin{align*}
            &\sigma_{i}(\bm{x})\in C^1(\Omega^l), \; \mbox {such that} \;\; \sigma_i(\bm{x})>0,  \\
            & A(\bm{x})=\{a_{ij}(\bm{x}) \} \; \mbox{with} \; a_{ij}(\bm{x})\in C^1(\Omega^l), \; \mbox {such that the matrix} \;  I-A(\bm{x}) \; \mbox{is invertible,}    
        \end{align*}  
        independently of $\bm{x}.$ Then, the functions 
        \begin{itemize}
            \item[] $Z: \bm{x}\in\Omega^l \to \diag\left(\bm{S}(\bm{x})\right) \in \mathbb{R}^{N\times N},$ with $\bm{S}(\bm{x})=\left(\frac{1}{\sigma_1(\bm{x})},\dots,\frac{1}{\sigma_N(\bm{x})}\right)^\mathsf{T}\in \mathbb{R}^{N},$
            \item[] $F: \bm{x}\in\Omega^l \to (I-A(\bm{x}))^{-1} \in \mathbb{R}^{N\times N},$
            \item[] $\bm{g}: \bm{x}\in\Omega^l \to F(\bm{x}) \sum_{r=1}^l\alpha_r \bm{x}^r \in \mathbb{R}^{N},$ % with $\bm{\alpha}\in \mathbb{R}^l,$
        \end{itemize}
        are continuously differentiable on $\Omega^l.$
    \end{lemma}
    \begin{proof}
         The first statement comes from 
         \begin{align*}
             &\dfrac{\partial Z}{\partial x^v_j}(\bm{x}^1,\dots,\bm{x}^l)=-\diag\left(\bar{\bm{S}}(\bm{x})\right) \diag\left(\hat{\bm{S}}^{vj}(\bm{x})\right), \qquad\quad
             \begin{array}{l}
                  j=1,\dots,N,  \\
                  v=1,\dots,l, 
             \end{array}
             \qquad \ \mbox{with}
         \\
             &\bar{\bm{S}}(\bm{x})=\left(\dfrac{1}{\sigma_1^2(\bm{x})},\dots,\dfrac{1}{\sigma_N^2(\bm{x})}\right)^\mathsf{T}\in \mathbb{R}^{N}, \;\;\; \hat{\bm{S}}^{vj}(\bm{x})=\left(\dfrac{\partial \sigma_1(\bm{x})}{\partial x^v_j},\dots,\dfrac{\partial \sigma_N(\bm{x})}{\partial x^v_j}\right)^\mathsf{T} \in \mathbb{R}^{N}.
         \end{align*}         
         Denote with $\bar{A}(\bm{x})$ the adjoint matrix of $I-A(\bm{x}),$ whose entries are continuous functions of the coefficients $a_{ij}(\bm{x})$. 
         Because of the assumptions, $\bar{A}(\bm{x})$ and $D_A : \bm{x} \in \Omega^l \to 1/\det{A(\bm{x})}\in \mathbb{R}$ are well posed and continuously differentiable functions. Therefore, from $F(\bm{x})=D_A(\bm{x})\bar{A}(\bm{x}),$ it follows $F\in C^1(\Omega^l).$ Finally, 
         \begin{equation*}
             \dfrac{\partial \bm{g}}{\partial x^v_j}(\bm{x}^1,\dots,\bm{x}^l)=\dfrac{\partial F}{\partial x^v_j}(\bm{x})\sum_{r=1}^l\alpha_r \bm{x}^r+\alpha_v (F_{1j}(\bm{x}),\dots,F_{Nj}(\bm{x}))^\mathsf{T} , \qquad
             \begin{array}{l}
                  j=1,\dots,N,  \\
                  v=1,\dots,l, 
             \end{array}
         \end{equation*}
         yields the result.
    \end{proof}

    The following result, whose proof is deducted from the arguments in \cite[sec. 3.4]{Lambert_Vecchio}, facilitates the analysis of the approximation error of \eqref{eq:Compact_Expression}.  

    \begin{lemma}\label{lem:Limite_Successione}
        Let $l\geq 1$ be a positive integer. Consider a sequence of non-negative numbers $\{a_n\}_{n\in \mathbb{N}_0}$ and assume that there exist $b\geq 0$ and $c_r\geq 0,$ $r=1,\dots,l,$ such that $c=\sum_{r=1}^l c_{r}> 1$ and  $a_n\leq b + \sum_{r=1}^l c_{r}a_{n-r}$ for $n\geq l.$ Then
        \begin{equation}\label{eq: bound_del_lemma}
            a_n\leq \left(a^*+\dfrac{b}{c-1}\right)\exp\left(n(c-1)\right), \qquad \qquad \qquad n\geq 0,
        \end{equation}  
        where $a^*=\max_{0\leq j \leq l-1}a_j.$
    \end{lemma}
    \begin{proof}
        For $n=0,\dots, l-1,$ the bound \eqref{eq: bound_del_lemma} directly follows from the definition of $a^*$ and the non-negativity of the termes involved. To prove it for $n\geq l,$ we firstly show by induction that  $a_n\leq c^na^*+b\sum_{j=0}^{n-1}c^j.$ When $n=l,$ from $a_l\leq ca^*+b$ and $l\geq 1,$ we get the result. Consider now $n>l$ and assume that the statement holds for $j=l,\dots, n-1.$ It follows that
        \begin{equation*}
                a_n\leq b+\sum_{r=1}^l c_{r}\left(c^{n-r}a^*+b\sum_{j=0}^{n-r-1}c^j\right)\leq b+\sum_{r=1}^l c_{r}\left(c^{n-1}a^*+b\sum_{j=0}^{n-2}c^j\right)\leq c^na^*+b\sum_{j=0}^{n-1}c^j.
        \end{equation*}
        Therefore, it turns out that $a_n\leq c^na^*+b (c^n-1)/(c-1)$ for $n\geq l.$ Finally $c^n\leq exp(n(c-1)),$ which completes the proof.
    \end{proof}

    A thorough analysis of the global discretization error of \eqref{eq:Compact_Expression} leads to the following convergence result. 
    \begin{theorem}\label{thm:Convergence}
		Let $\bm{y}(t)$ be the continuous-time solution to \eqref{eq:sistema_compatto} for $t\in [0,T],$ with $T>0$ and let $\{\bm{y}^{n}\}_{n\geq 0}$ be its approximation computed by the $k$-steps MPLM scheme \eqref{eq:Compact_Expression} with $h=T/\bar{n}.$ Define
        \begin{equation}\label{eq:def_Omega}
            \Omega=\left\{\bm{x}\in \mathbb{R}^N : \mu \leq x_i\leq \bm{e}^\mathsf{T}\bm{y}^0,\; i=1,\dots,N\right\},
        \end{equation} 
        % \underset{0\leq n \leq \bar{n}}{\min}\{y_i(t_n),y_i^n\}
        with $\mu$ positive constant. Assume that the given functions describing problem \eqref{eq:PDS} belong to $C^p(\Omega),$ with $p\geq 1$ and that
		\begin{itemize}
			\item the starting values satisfy $\|\bm{y}(t_m)-\bm{y}^m\|=\mathcal{O}(h^p),$ $m=0,\dots,k-1;$
            \item the PWDs are continuously differentiable functions on $\Omega^k$ and satisfy  \eqref{eq:cond_sigma}.
		\end{itemize}
		Then, the method \eqref{eq:Compact_Expression} is convergent of order $p.$
	\end{theorem}

    \begin{proof}
        Because of the properties of the continuous and the numerical solution to \eqref{eq:sistema_compatto} outlined in \eqref{eq: pos_cont} and Lemma \ref{lem:positivity}, there exists $\mu>0$ such that $\bm{y}(t)$ and $\bm{y}^{n}$ belong to $\Omega,$ for all $t\geq0$ and $n=0,\dots,\bar{n}.$
        Define, for $\bm{x}=(\bm{x}^1,\dots,\bm{x}^k)\in \Omega^k$ and $\bm{S}^n(\bm{x})=\left(\frac{1}{\sigma_1^n(\bm{x})},\dots,\frac{1}{\sigma_N^n(\bm{x})}\right)^\mathsf{T}\in \mathbb{R}^{N},$ the functions 
        \begin{equation*}%\label{eq:Funzioni_MeG_Consistenza}
            \begin{split}
                \Phi^{nr} &: \bm{x}\in \Omega^k \ \to \ (P(\bm{x}^{r})-\diag(D(\bm{x}^{r})\bm{e})) \diag\left(\bm{S}^n(\bm{x}) \right)\in \mathbb{R}^{N\times N} \\
                \bm{G}^{nr} &:  (\bm{x},\bm{x}^{k+1})\in \Omega^{k+1} \ \to \ \Phi^{nr}(\bm{x})\ \bm{x}^{k+1}\in \mathbb{R}^N, \qquad \qquad \qquad \begin{array}{r}
                     r=1,\dots,k,  \\
                     n=k,\dots,\bar{n}.
                \end{array}
             \end{split}
         \end{equation*}
        The global discretization error $\bm{e}(h;t_n)=\bm{y}(t_n)-\bm{y}^{n}$ then satisfies 
        \begin{equation}\label{eq:global_err_1}
        \begin{split}            
            \bm{e}(h;t_n)=&\sum_{r=1}^k\alpha_r\bm{e}(h;t_{n-r})+h\sum_{r=1}^k\beta_r \left(\bm{G}^{nr}(\bm{y}(t_{n-k}),\dots,\bm{y}(t_{n}))-\bm{G}^{nr}(\bm{y}^{n-k},\dots,\bm{y}^{n}) \right) \\
            &+\bm{\delta}(h;t_n), \qquad \qquad \qquad\qquad \qquad \qquad\qquad \qquad \qquad\qquad \!\!  n=k,\dots,\bar{n},
        \end{split}
        \end{equation}
        where the components of $\bm{\delta}(h;t_n)=(\delta_1(h;t_n),\dots,\delta_N(h;t_n))^\mathsf{T}\in \mathbb{R}^N$ are defined in \eqref{eq:Local_Error}. Because of the regularity assumptions on the known functions and the first result of Lemma \ref{lem:diff}, $\Phi^{nr}\in C^1(\Omega^k)$ and $\bm{G}^{nr}\in C^1(\Omega^{k+1}),$ for each $r=1,\dots,k$ and $n=k,\dots,\bar{n}.$ Therefore, from the mean value theorem 
        \begin{equation}\label{eq:global_err_2}
            \begin{split}
                \bm{G}^{nr}(\bm{y}(t_{n-k}),\dots,\bm{y}(t_{n}))-\bm{G}^{nr}(\bm{y}^{n-k},\dots,\bm{y}^{n})&=J_{\bm{G}^{nr}}(\bm{\xi}^{n-k}_r,\dots,\bm{\xi}^{n}_r)
                \begin{pmatrix}
                    \bm{e}(h;t_{n-k})  \\ \vdots \\ \bm{e}(h;t_n)
                \end{pmatrix},
            \end{split}
        \end{equation}
        where 
        $J_{\bm{G}^{nr}}(\bm{\xi}^{n-k}_r,\dots,\bm{\xi}^{n}_r)\in \mathbb{R}^{N\times (Nk+N)}
        $ is the Jacobian of $\bm{G}^{nr}.$ % whose rows are evaluated at different mean values $\bm{\xi}^n_r,\dots,\bm{\xi}^{n-k}_r\in \mathring{\Omega},$ for $r=1,\dots,k$ and $n=k,\dots, \bar{n}.$ 
        Let %, for $n=k,\dots, \bar{n}$ and $r=1,\dots, k,$
        \begin{equation*}
            J_{\bm{G}^{nr}}(\bm{\xi}^{n-k}_r,\dots,\bm{\xi}^{n}_r)=\bigg( \underbrace{ J_{\bm{G}^{nr}}^{(0)}(\bm{\xi}^{n-k}_r,\dots,\bm{\xi}^{n}_r)}_{N\times N} \; \bigg| \; \dots \; \bigg| \; \underbrace{J_{\bm{G}^{nr}}^{(k)}(\bm{\xi}^{n-k}_r,\dots,\bm{\xi}^{n}_r)}_{N\times N} \bigg).
        \end{equation*}
        % \begin{equation*}
        %     J_{\bm{G}^{nr}}(\bm{\xi}^{n-k}_r,\dots,\bm{\xi}^{n}_r)=(J_{\bm{G}^{nr}}^{(0)}(\bm{\xi}^{n-k}_r,\dots,\bm{\xi}^{n}_r),\dots,J_{\bm{G}^{nr}}^{(k)}(\bm{\xi}^{n-k}_r,\dots,\bm{\xi}^{n}_r)), \qquad n=k,\dots, \bar{n},
        % \end{equation*}
        % where $J_{\bm{G}^{nr}}^{(j)}(\bm{\xi}^{n-k}_r,\dots,\bm{\xi}^{n}_r)\in \mathbb{R}^{N\times N}$ is a submatrix of $J_{\bm{G}^{nr}}(\bm{\xi}^{n-k}_r,\dots,\bm{\xi}^{n}_r).$ 
        Then, for each $n=k,\dots, \bar{n}$ and $r=1,\dots,k,$
        \begin{equation}\label{eq: J_bound}
            \|J_{\bm{G}^{nr}}^{(j)}(\bm{\xi}^{n-k}_r,\dots,\bm{\xi}^{n}_r)\|\leq \|J_{\bm{G}^{nr}}(\bm{\xi}^{n-k}_r,\dots,\bm{\xi}^{n}_r)\| \leq J, \qquad\quad j=0,\dots,k,
        \end{equation}
        with $J$ positive constant depending on the bounds of the known functions and their derivatives on $\Omega^{k+1}.$ Substituting \eqref{eq:global_err_2} into \eqref{eq:global_err_1} leads to the discrete equation
        \begin{equation*}
                \bm{e}(h;t_n)=\sum_{r=1}^k\alpha_r\bm{e}(h;t_{n-r})+h\sum_{r=1}^k\beta_r\sum_{j=0}^k J_{\bm{G}^{nr}}^{(j)}(\bm{\xi}^{n-k}_r,\dots,\bm{\xi}^{n}_r) \
                \bm{e}(h;t_{n-j})+\bm{\delta}(h;t_n), 
        \end{equation*}
        for $n=k,\dots, \bar{n}.$  Denoted $B=J\sum_{r=1}^k\beta_r,$ for a sufficiently small $h$, from \eqref{eq: J_bound}
        \begin{equation*}
                \|\bm{e}(h;t_n)\|\leq \sum_{r=1}^k\dfrac{\alpha_r+hB}{1-hB}\|\bm{e}(h;t_{n-r})\|+\underset{0\leq n \leq \bar{n}}{\max}\dfrac{\|\bm{\delta}(h;t_n)\|}{1-hB}, \qquad \quad \quad n=k,\dots, \bar{n}. 
        \end{equation*}
        Because of the first relation in \eqref{eq:Pre_Consistency}, $\sum_{r=1}^k\frac{\alpha_r+hB}{1-hB}=1+h \frac{(k+1)B}{1-hB}> 1,$ so that from Lemma \ref{lem:Limite_Successione} it follows 
        \begin{equation*}
            \|\bm{e}(h;t_n)\|\leq \left(\underset{0\leq m \leq k-1}{\max}\|\bm{y}(t_m)-\bm{y}^m \| + \dfrac{\max_{0\leq n \leq \bar{n}}\|\bm{\delta}(h;t_n)\|}{h \ (k+1)B} \right) \exp\left(\frac{(k+1)TB}{1-hB}\right),
        \end{equation*}
        for $0\leq n \leq \bar{n}$ and $T=\bar{n}h.$ Because of the assumptions on the initial values and the order $p$ consistency of \eqref{eq:Compact_Expression} (Theorem \ref{thm:Consis_order_p}), we have 
        \begin{equation*}
             \max_{0\leq n \leq \bar {n}}\|\bm{e}(h;t_n)\| \leq C h^p,
         \end{equation*}
        with $C$ positive constant not depending on $h,$ which yields the result.
        % there exist two positive constants $\hat{C}$ and $\bar{C}$ such that
        % \begin{equation*}
        %     \underset{0\leq m \leq k-1}{\max}\|\bm{y}(t_m)-\bm{y}^m \|\leq \hat{C}h^p \qquad \;\; \mbox{and } \;\; \qquad \underset{0\leq n \leq \bar{n}}{\max}\|\bm{\delta}(h;t_n)\|\leq \bar{C}h^{p+1}.
        % \end{equation*}
        % The result then comes from 
        % \begin{equation*}
        %     \max_{0\leq n \leq \bar {n}}\|\bm{e}(h;t_n)\| \leq \left( \left(\hat{C}+\dfrac{\bar{C}}{(k+1)B}\right) \exp\left(\frac{(k+1)TB}{1-hB}\right)\right) h^p.
        % \end{equation*}
    \end{proof}
    
    \subsection{The $\sigma$-embedding technique} \label{subsec:sigma_embedding}
     Theorem \ref{thm:Convergence} provides sufficient conditions for the order $p$ convergence of the numerical method \eqref{eq:Compact_Expression}, but give no clues on how to compute unconditionally positive and conservative Patankar weight denominators satisfying \eqref{eq:cond_sigma}. 
     To achieve this goal, here we introduce an embedding technique based on a recursive use of MPLM methods. More specifically, for $p=1,$ we just consider the Modified Patankar Euler (MPE) method (see \cite{MPEuler} for further details), corresponding to \eqref{eq:Compact_Expression} with
    \begin{equation}\label{eq:MPEul}
        k=1, \qquad \alpha_1=\beta_1=1, \qquad \sigma_i^n=y_i^{n-1}, \quad \mbox{for} \quad n\geq 1, \quad i=1,\dots,N,
    \end{equation}
    and then investigate MPLM-$k$ schemes for $p\geq 2$ and $k>1.$ Let $\bm{\sigma}^{n}=\bm{\sigma}^{n(p-1)}$ be PWDs satisfying the condition \eqref{eq:cond_sigma}. Our approach consists in recursively computing $\bm{\sigma}^{n(p-1)}$ by a $\tilde{k}\leq k$ steps, order $p-1$ convergent MPLM-$\tilde{k}$ method, whose coefficients are denoted by $\bm{\alpha}^{(p-1)}\in \mathbb{R}^{\tilde{k}}$ and $\bm{\beta}^{(p-1)}\in \mathbb{R}^{\tilde{k}}$, as follows
    \begin{equation}\label{eq:sigma_embedding}
        \bm{\sigma}^{n(p-1)}(\bm{y}^{n-1 (p)},\dots,\bm{y}^{n-\tilde{k}(p)})= \left(I-h\sum_{r=1}^{\tilde{k}}\beta_r^{(p-1)}M^{n-r(p)}\right)^{-1}  \sum_{r=1}^{\tilde{k}}\alpha_r^{(p-1)}\bm{y}^{n-r (p)}, 
    \end{equation}
    for $n\geq k,$ where
    \begin{equation*}
        M^{n-r(p)}=\left( P(\bm{y}^{n-r(p)})-\diag(D(\bm{y}^{n-r(p)})\bm{e}) \right) \diag(\bm{S}^{n(p-2)})\in \mathbb{R}^{N\times N}, \quad r=1,\dots,k
    \end{equation*}
    and $\bm{S}^{n(p-2)}=(1/\sigma_1^{n(p-2)},\dots,1/\sigma_N^{n(p-2)})^\mathsf{T}\in \mathbb{R}^{N}.$ Here $\bm{y}^{n(p)},$ $n\geq k,$ represents the numerical solution computed by the MPLM-$k$ method.

    With the following result we prove that the PWDs computed adopting the $\sigma$-embedding technique \eqref{eq:sigma_embedding} meet the condition \eqref{eq:cond_sigma}, which is sufficient for to the convergence of the MPLM-$k$ scheme.
    \begin{theorem}\label{thm:sigma_embedding}
        Let $\bm{y}(t)$ be the continuous-time solution to \eqref{eq:sistema_compatto} for $t\in [0,T],$ with $T>0$ and let $\{\bm{y}^{n(p)}\}_{n\geq 0}$ be its approximation computed by the $k$-steps MPLM scheme \eqref{eq:Compact_Expression} with $h=T/\bar{n}.$ Assume that the given functions describing problem \eqref{eq:PDS} belong to $C^p(\Omega),$ with $p\geq 1$ and $\Omega$ in \eqref{eq:def_Omega}. % Define
        %\begin{equation*}
        %    m=\frac{1}{2} \  \underset{0\leq t \leq T}{\min} \left( \underset{1\leq i \leq N}{\min} y_i(t) \right) \qquad \mbox{and} \qquad \hat{\Omega}=\prod_{i=1}^{N}\left[m,\bm{e}^\mathsf{T}\bm{y}^0\right].
        % \end{equation*}
        Assume that
        \begin{itemize}
        \item the starting values satisfy $\|\bm{y}(t_m)-\bm{y}^{m(p)}\|=\mathcal{O}(h^p),$ $m=0,\dots,k-1;$
            \item the $\sigma$-embedding strategy is implemented and the PWDs are computed by a $\tilde{k}$-steps, order $p-1$ convergent MPLM-$\tilde{k}$ scheme as detailed in \eqref{eq:sigma_embedding}.
            % \item the MPLM-$k$ coefficients satisfy \ref{hp':1}, \eqref{eq:Order_Condition_LM} and the root condition.
        \end{itemize}
        Then, the PWDs functions $\bm{\sigma}^{n(p-1)},$ $n=k,\dots,\bar{n},$ are continuously differentiable on $\Omega^{\tilde{k}}$ and satisfy \eqref{eq:cond_sigma}. Therefore, the numerical method \eqref{eq:Compact_Expression} is convergent of order $p.$
    \end{theorem}
    \begin{proof}
        The case $p=1$ corresponds to the first order convergent MPE method \eqref{eq:MPEul}. For $p\geq 2,$ we prove the result by induction. Firstly, we show that the MPLM-$k$ scheme \eqref{eq:Compact_Expression} is quadratically convergent ($p=2$) if the PWDs are computed by the one step Modified Patankar Euler discretization \eqref{eq:MPEul}. As a matter of fact, the consistency condition \eqref{eq:cond_sigma} comes from \cite[Theorem 3.11]{MPEuler}, since
        \begin{equation*}
            \bm{\sigma}^{n(1)}(\bm{y}(t_{n-1}))= \left(I-h\Phi(\bm{y}(t_{n-1}))\right)^{-1}  \bm{y}(t_{n-1})=\bm{y}(t_{n})+\mathcal{O}(h^2), \qquad n\geq 1,
            % \bm{\sigma}^{n(1)}(\bm{y}^{n-1 (2)})= \left(I-hM^{n-1(2)}\right)^{-1}  \bm{y}^{n-1 (2)}
        \end{equation*}
        where $\Phi(\bm{x})=(P(\bm{x})-\diag(D(\bm{x})\bm{e})) \diag\left(\frac{1}{x_1},\dots,\frac{1}{x_N}\right)\in \mathbb{R}^{N\times N}.$ Furthermore, from Lemma \ref{lem:diff}, $\bm{\sigma}^{n(1)}\in C^1(\Omega).$ Therefore, all the hypotheses of Theorem \ref{thm:Convergence} are fulfilled and the second order convergence is established. \\
        Consider $p>2$ and assume the statement to be true for each $s=1,\dots,p-1.$ In this case, the PWDs $\bm{\sigma}^{n(p-1)},$ $n=k,\dots,\bar{n},$ are computed by an order $p-1$ convergent, then consistent \cite[Theorem 3.5]{Chartres}, MPLM-$\tilde{k}$ method, accordingly to \eqref{eq:sigma_embedding}. Therefore, from Theorem \ref{Them_Con_Neces}, the condition \eqref{eq:cond_sigma} directly follows. Furthermore, from the inductive hypotheses, the functions $\bm{\sigma}^{n(p-2)}\in C^1(\Omega^{\tilde{k}}),$ $n=k,\dots,\bar{n}$ and hence, from Lemma \ref{lem:diff}, 
        \begin{align*}
            &\bm{\sigma}^{n(p-1)}(\bm{x})=\left(I-h\sum_{r=1}^{\tilde{k}}\beta_r^{(p-1)}\Phi^{nr(p-2)}(\bm{x})\right)^{-1}  \sum_{r=1}^{\tilde{k}}\alpha_r^{(p-1)}\bm{x}^{r}\in \mathbb{R}^{N}, \\
            &\mbox{with} \;\;\; \Phi^{nr(p-2)}(\bm{x})=(P(\bm{x}^{r})-\diag(D(\bm{x}^{r})\bm{e})) \diag\left(\bm{S}^{n(p-2)}(\bm{x})\right)\in \mathbb{R}^{N\times N}, \\
            &\mbox{and} \;\;\;\; \bm{S}^{n(p-2)}(\bm{x})=\left(\dfrac{1}{\sigma_1^{n(p-2)}(\bm{x})},\dots,\dfrac{1}{\sigma_N^{n(p-2)}(\bm{x})}\right)^\mathsf{T}\in \mathbb{R}^{N}, \quad \bm{x}=(\bm{x}^1,\dots,\bm{x}^{\tilde{k}})^\mathsf{T},
        \end{align*}
        is continuously differentiable on $\Omega^{\tilde{k}}$ as well. Finally, an application of Theorem \ref{thm:Convergence} yields the result.
    \end{proof}
Theorem \ref{thm:sigma_embedding} establishes a practical and general framework to compute the PWDs by embedding different MPLM methods, starting from the modified Patankar Euler discretization. This machinery allows the construction of arbitrarily high order unconditionally positive and conservative numerical methods. However, the $\sigma$-embedding technique in \eqref{eq:sigma_embedding} is specifically designed for positive PDS and cannot be implemented if any of the components of the initial value is zero. This issue has been addressed and overcome in \cite{Kopecz2018} for modified Patankar Runge-Kutta schemes by replacing zero components of the initial value by small quantities like $\textit{realmin}\approx2.26\cdot 10^{-308}$. Our numerical experiments of Section \ref{sec:numerical_Experiments} confirm the effectiveness of this expedient also for the modified Patankar linear multistep methods. 
% Anyway eq. (3.12) allow us to add a positive term, such as $h^p$, so that the PWDs cannot vanish even in case of a zero starting value.

\begin{table} \center \footnotesize
\setlength{\tabcolsep}{2.0pt}
% table caption is above the table
\caption{Coefficients of the MPLM-$k(p)$ methods.}
\label{tab:1}       % Give a unique label
\smallskip
% For LaTeX tables use
\renewcommand{\arraystretch}{0.75} % DEFAULT IS 1
\begin{tabular}{ccccccccccc}
\hline\noalign{\smallskip}\noalign{\smallskip}
\multicolumn{11}{c}{\small MPLM-$2(2)$}  \\
\noalign{\smallskip}
\hline\noalign{\smallskip}
$r$ & $1$ & $2$ \\
$\alpha_r$ & $0$ & $1$ &  &  &  &  &  &  &  &  \\
$\beta_r$ & $2$ & $0$ &  &  &  &  &  &  &  &  \\
\noalign{\smallskip} 

\hline\noalign{\smallskip} \noalign{\smallskip}
\multicolumn{11}{c}{\small MPLM-$4(3)$}  \\
\noalign{\smallskip}
\hline\noalign{\smallskip}
$r$ & $1$ & $2$ & $3$ & $4$ \\
$\alpha_r$ & $1/4$ & $0$ & $3/4$ & $0$ &  & &  &  &  &  \\
$\beta_r$ & $35/18$ & $1/3$ & $0$ & $2/9$ &  & &  &  &  &  \\
\noalign{\smallskip}

\hline\noalign{\smallskip} \noalign{\smallskip}
\multicolumn{11}{c}{\small MPLM-$5(4)$}  \\
\noalign{\smallskip}
\hline\noalign{\smallskip}
$r$ & $1$ & $2$ & $3$ & $4$ & $5$  \\
$\alpha_r$ & $0$ & $0$ & $0$ & $0$ & $1$ & &  &  &  &  \\
$\beta_r$ & $75/32$ & $0$ & $25/48$ & $25/12$ &  $5/96$ & &  &  &  &  \\
\noalign{\smallskip}

\hline\noalign{\smallskip} \noalign{\smallskip}
\multicolumn{11}{c}{\small MPLM-$7(5)$}  \\
\noalign{\smallskip}
\hline\noalign{\smallskip}
$r$ & $1$ & $2$ & $3$ & $4$ & $5$ & $6$ & $7$  \\
$\alpha_r$ & $0$ & $0$ & $0$ & $0$ & $0$ & $0$ & $1$ &  &  &  \\
$\beta_r$ & $12/5$ & $0$ & $197/720$ & $701/360$ &  $43/30$ & $107/360$ & $467/720$ &  &  &  \\
\noalign{\smallskip}

\hline\noalign{\smallskip} \noalign{\smallskip}
\multicolumn{11}{c}{\small MPLM-$10(6)$}  \\
\noalign{\smallskip}
\hline\noalign{\smallskip}
$r$ & $1$ & $2$ & $3$ & $4$ & $5$ & $6$ & $7$ & $8$ & $9$ & $10$  \\
$\alpha_r$ & $0$ & $0$ & $0$ & $0$ & $0$ & $0$ & $0$ & $0$ & $0$ & $1$  \\
$\beta_r$ & $11125/4536$ & $0$ & $0$ & $50/27$ &  $85/36$ & $0$ & $0$ & $125/63$ & $25/24$ &  $25/81$  \\
\noalign{\smallskip}\hline

\end{tabular}
\end{table}

\section{Numerical Experiments}\label{sec:numerical_Experiments}
In this section we report some experiments and compare the performances of the method \eqref{eq:Compact_Expression} with that of some well-established modified Patankar schemes. Our investigation is conducted employing the modified Patankar linear multistep methods of order from $2$ to $6$, whose coefficients are reported in Table \ref{tab:1}. Here, we adopt the notation MPLM-$k(p)$ to indicate a $k$-steps, order $p$  scheme. The numerical simulations of some test problems performed by the methods \eqref{eq:Compact_Expression}-Table \ref{tab:1} with the $\sigma$-embedding technique provide experimental evidence of convergence up to the sixth order. A self-starting embedding procedure is here implemented for computing accurate starting values satisfying \ref{hp:3} and \ref{hp':2}.
To assess the order of the MPLM-$k(p)$ schemes, we consider the maximum absolute error $E(h)$ and the experimental rate of convergence $\hat{p}$ defined as follows
\begin{equation}\label{eq:E(h)}
    E(h)=\max_{0\leq n \leq T/h} \left\|\bm{y}^{n(ref)}-\bm{y}^n\right\|_\infty, \qquad \quad \hat{p}=\log_2\left(\dfrac{E(h)}{E\left(\frac{1}{2}h\right)}\right).
\end{equation}
For each test problem, the reference solution $\bm{y}^{n(ref)}$ in \eqref{eq:E(h)} is obtained by the \texttt{ode15s} built-in Matlab function with absolute and relative tolerances of $\texttt{AbsTol}=\varepsilon_{mach}\approx2.22\cdot 10^{-16}$ and $\texttt{RelTol}=\varepsilon_{mach}\cdot 10^2 $, respectively. 

A direct comparison with a third order Modified Patankar Runge-Kutta (MPRK$3$) method (see \cite[Lemma $6,$ Case II]{Kopecz2018} with $\gamma=0.5$) and with high order Modified Patankar Deferred Correction (MPDeC) schemes (see \cite{Torlo2020}), highlights the competitive performances of the MPLM-$k$ discretizations.  Incidentally, in the case of the MPDeC integrators, the computational efforts required to obtain the coefficients are disregarded.

All the numerical experiments are conducted on a single machine equipped with an Intel Core i7-7700HQ Octa-Core processor operating at 2.80GHz and supported by 8.00 GB of RAM. Both MPRK$3$ and MPLM-$k$ algorithms are implemented and executed using MATLAB (version R2020b). Additionally, the MPLM-$k$ schemes are implemented and executed with Julia (version 1.8.5) and compared against the MPDeC codes available in the repository \cite{Torlo_Rep}.

\subsection{Test 1: linear test}\label{Test1}
Our first example consists in the following linear test
		\begin{equation}\label{eq:Linear_Test}
			\begin{aligned}
                y^{\prime}_1(t)=-a y_1(t)+y_2(t), \\
                y^{\prime}_2(t)=\phantom{-}a y_1(t)-y_2(t),
            \end{aligned}  \qquad \mbox{with} \quad a=5, \quad t\in[0,2] \quad \mbox{and} \quad  \bm{y}^0=\begin{pmatrix}0.9 \\ 0.1 \end{pmatrix},
		\end{equation}
proposed in \cite{Kopecz2018,Kopecz2018second,Huang2019,Huang2019II}. The system \eqref{eq:Linear_Test} describes the exchange of mass between two constituents and fits the form of a fully conservative PDS, where the production and destruction terms in \eqref{eq:PDS_terms} are given by
		\begin{equation*}
			p_{12}(\bm{y})=y_2, \qquad p_{21}(\bm{y})=ay_1, \qquad d_{12}(\bm{y})=ay_1, \qquad d_{21}(\bm{y})=y_2.
		\end{equation*}
Since $p_{ij}(\bm{y})$ is continuously differentiable with bounded derivatives and $\lim_{\bm{y}\to \bm{0}}p_{ij}(\bm{y})=0,$ for $i,j\in \{1,2\},$ the existence of a unique and positive solution to \eqref{eq:Linear_Test} is guaranteed by \cite[Theorem 1.2]{TORLO2022}.
\begin{figure}
	\begin{center}
		\scalebox{0.85}{\includegraphics{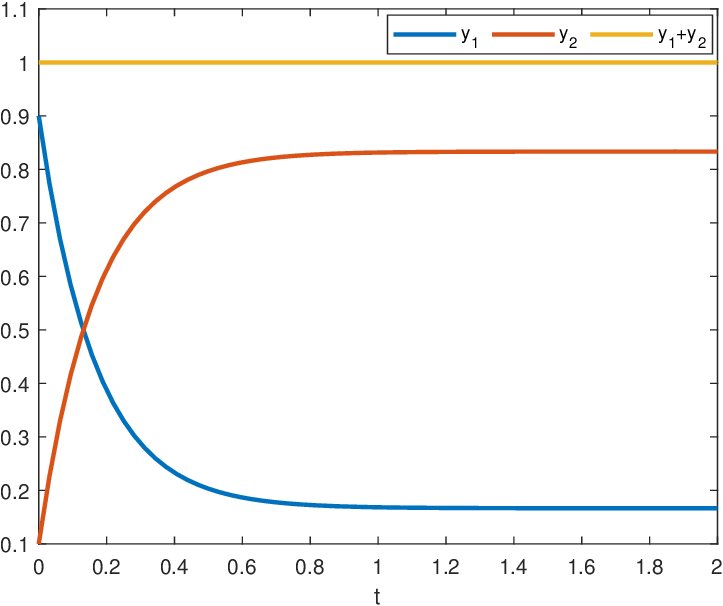}}
	\end{center}%\vspace{-0.5cm}
	\caption{Numerical solution of \hyperref[Test1]{Test 1} computed by MPLM-$10(6)$ with $h=2^{-5}$ .} %For comparison purposes Numerical solution obtained by MPLM}
\label{fig:Linear_Test_PL}
\end{figure}

The approximation of the continuous-time solution to \eqref{eq:Linear_Test}, computed by the MPLM-$10(6)$ method with $h=2^{-5},$ is reported in Figure \ref{fig:Linear_Test_PL}.
In compliance with Lemmas \ref{lem:positivity} and \ref{lem:conservativity}, the numerical solution is positive and the linear invariant of the PDS \eqref{Test1} is retained. 
In Table \ref{tab:Linear_Test_Err} we list the maximum errors $E(h)$ on the integration interval $[0,2]$ and the experimental rate of convergence $\hat{p}$ for all the methods listed in Table \ref{tab:1}. From Table \ref{tab:Linear_Test_Err}, as well as from Figures \ref{fig:Linear_Test_1} and \ref{fig:Linear_Test_2}, it is clear that the experimental order agrees with the theoretical one established in Theorem \ref{thm:sigma_embedding}. Furthermore, the work precision diagram of Figure \ref{fig:linear_WPD} shows the mean execution time over $10$ runs against the approximation error for the MPLM-$k(p)$ methods with $2\leq p\leq 6.$ It is evident that, for $p\geq 4,$ the methods in Table \ref{tab:1} are computationally more efficient than the benchmark scheme MPRK3, in terms of the accuracy-computational cost trade off. 

\begin{figure}
	\begin{center}
		\scalebox{1}{\includegraphics{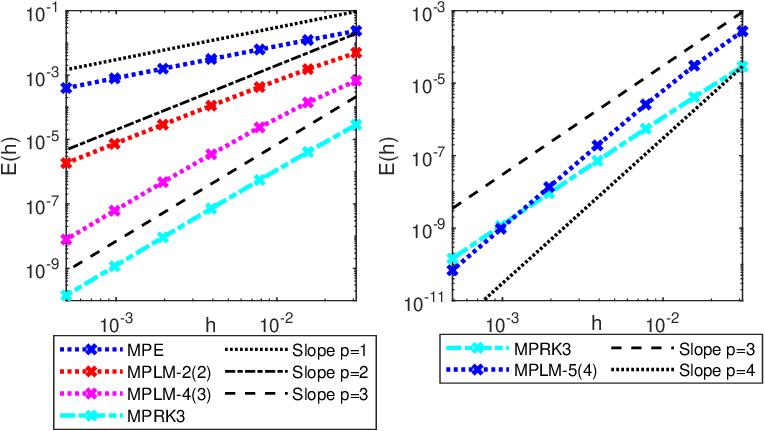}}
	\end{center}%\vspace{-0.65cm}
	\caption{Experimental order for MPE, MPLM-$2(2),$ MPLM-$4(3),$ MPLM-$5(4)$ and MPRK3 applied to \hyperref[Test1]{Test 1}.} %For comparison purposes Numerical solution obtained by MPLM}
\label{fig:Linear_Test_1}
\end{figure}
\begin{figure}
	\begin{center}
		\scalebox{1}{\includegraphics{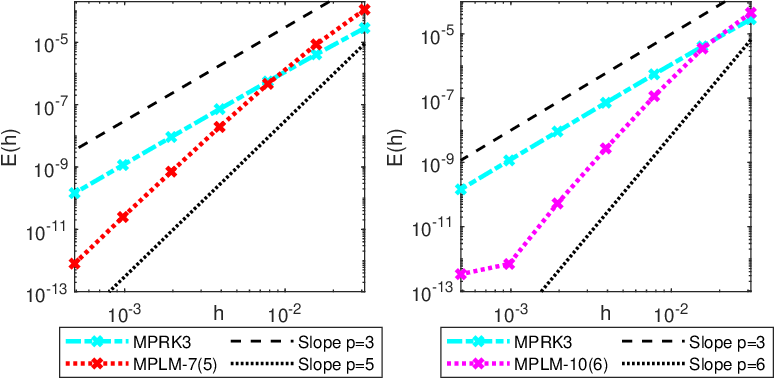}}
	\end{center}%\vspace{-0.65cm}
	\caption{Experimental order for MPLM-$7(5),$ MPLM-$10(6)$ and MPRK3 applied to \hyperref[Test1]{Test 1}.} %For comparison purposes Numerical solution obtained by MPLM}
\label{fig:Linear_Test_2}
\end{figure}
\begin{figure}
\begin{center}
	\scalebox{1}{\includegraphics{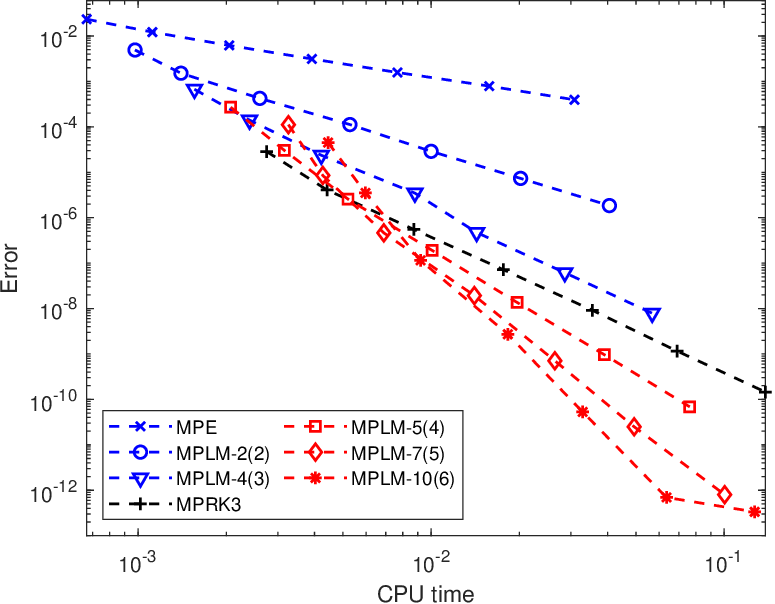}}
\end{center}%\vspace{-0.65cm}
\caption{Work Precision Diagram: Error versus CPU time for the different methods applied to \hyperref[Test1]{Test 1}. $h=T/2^{5+m}$, $m=1,\ldots,7.$}
\label{fig:linear_WPD}
\end{figure}

\begin{table} \center \footnotesize
% table caption is above the table
\caption{Experimental convergence of the numerical solutions to \hyperref[Test1]{Test 1}.}
\label{tab:Linear_Test_Err}       % Give a unique label
\smallskip
% For LaTeX tables use
\begin{tabular}{ccc|ccc|ccc}

\hline\noalign{\smallskip}
\multicolumn{3}{c|}{MPE} & \multicolumn{3}{|c}{MPLM-$2(2)$} & \multicolumn{3}{|c}{MPLM-$4(3)$}  \\
\noalign{\smallskip}\hline\noalign{\smallskip}
$h$ & $E(h)$ & $\hat{p}$ & $h$ & $E(h)$ & $\hat{p}$ & $h$ & $E(h)$ & $\hat{p}$ \\
$2^{-5}$ & $2.34 \cdot 10^{-2}$ & --- & $2^{-5}$ & $4.92 \cdot 10^{-3}$ & --- & $2^{-5}$ & $6.71 \cdot 10^{-4}$ & --- \\
$2^{-6}$ & $1.21 \cdot 10^{-2}$ & $0.94$ & $2^{-6}$ & $1.52 \cdot 10^{-3}$ & $1.70$ & $2^{-6}$ & $1.41 \cdot 10^{-4}$ & $2.25$ \\
$2^{-7}$ & $6.20 \cdot 10^{-3}$ & $0.97$ & $2^{-7}$ & $4.24 \cdot 10^{-4}$ & $1.84$ & $2^{-7}$ & $2.37 \cdot 10^{-5}$ & $2.57$ \\
$2^{-8}$ & $3.13 \cdot 10^{-3}$ & $0.98$ & $2^{-8}$ & $1.12 \cdot 10^{-4}$ & $1.92$ & $2^{-8}$ & $3.48 \cdot 10^{-6}$ & $2.77$ \\
$2^{-9}$ & $1.57 \cdot 10^{-3}$ & $0.99$ & $2^{-9}$ & $2.89 \cdot 10^{-5}$ & $1.96$ & $2^{-9}$ & $4.72 \cdot 10^{-7}$ & $2.88$ \\
$2^{-10}$ & $7.88 \cdot 10^{-4}$ & $1.00$ & $2^{-10}$ & $7.34 \cdot 10^{-6}$ & $1.98$ & $2^{-10}$ & $6.16 \cdot 10^{-8}$ & $2.94$ \\
$2^{-11}$ & $3.95 \cdot 10^{-4}$ & $1.00$ & $2^{-11}$ & $1.85 \cdot 10^{-6}$ & $1.99$ & $2^{-11}$ & $7.87 \cdot 10^{-9}$ & $2.97$ \\
\noalign{\smallskip} \hline

\noalign{\smallskip}
\multicolumn{3}{c}{MPLM-$5(4)$} & \multicolumn{3}{|c}{MPLM-$7(5)$} & \multicolumn{3}{|c}{MPLM-$10(6)$}  \\
\noalign{\smallskip}\hline\noalign{\smallskip}
$h$ & $E(h)$ & $\hat{p}$ & $h$ & $E(h)$ & $\hat{p}$ & $h$ & $E(h)$ & $\hat{p}$ \\
$2^{-5}$ & $2.70 \cdot 10^{-4}$ & --- & $2^{-5}$ & $1.12 \cdot 10^{-4}$ & --- & $2^{-5}$ & $4.52 \cdot 10^{-5}$ & --- \\
$2^{-6}$ & $3.02 \cdot 10^{-5}$ & $3.16$ & $2^{-6}$ & $8.53 \cdot 10^{-6}$ & $3.71$ & $2^{-6}$ & $3.51 \cdot 10^{-6}$ & $3.69$ \\
$2^{-7}$ & $2.57 \cdot 10^{-6}$ & $3.56$ & $2^{-7}$ & $4.64 \cdot 10^{-7}$ & $4.20$ & $2^{-7}$ & $1.15 \cdot 10^{-7}$ & $4.93$ \\
$2^{-8}$ & $1.91 \cdot 10^{-7}$ & $3.75$ & $2^{-8}$ & $1.93 \cdot 10^{-8}$ & $4.59$ & $2^{-8}$ & $2.71 \cdot 10^{-9}$ & $5.40$ \\
$2^{-9}$ & $1.36 \cdot 10^{-8}$ & $3.81$ & $2^{-9}$ & $7.09 \cdot 10^{-10}$ & $4.77$ & $2^{-9}$ & $5.30 \cdot 10^{-11}$ & $5.68$ \\
$2^{-10}$ & $9.63 \cdot 10^{-10}$ & $3.82$ & $2^{-10}$ & $2.49 \cdot 10^{-11}$ & $4.83$ & $2^{-10}$ & $6.95 \cdot 10^{-13}$ & $6.26$ \\
$2^{-11}$ & $6.88 \cdot 10^{-11}$ & $3.81$ & $2^{-11}$ & $7.98 \cdot 10^{-13}$ & $4.97$ & $2^{-11}$ & $3.34 \cdot 10^{-13}$ & $1.06$ \\
\noalign{\smallskip} \hline

\end{tabular}
\end{table}

The MPLM methods in Table \ref{tab:1} are then also compared to the modified Patankar deferred correction integrators outlined in \cite{Torlo2020}, utilizing the Julia implementation of the MPDeC schemes provided in \cite{Torlo_Rep}. For the sake of comparison, we consider the mean relative error taken over all time steps and all constituents,
\begin{equation}\label{eq:Relative_Error_all_components}
    \varepsilon(h) = \frac{1}{N} \sum_{i=1}^{N} \left( \dfrac{\sqrt{\bar{n} \sum_{n=1}^{\bar{n}}(y_i^{n(J)}-y_i^n)^2}}{\bar{n}\sum_{n=1}^{\bar{n}}y_i^{n(J)}} \right), \qquad\qquad \bar{n}=T/h,
\end{equation}
as defined in \cite{Kopecz2018}. Here, the reference solution $\bm{y}^{n(J)}$ in \eqref{eq:E(h)} is obtained by the \texttt{radau} built-in Julia solver (implicit Runge-Kutta of variable order between 5 and 13) with absolute and relative tolerances of $10^{-14}.$ It turns out that for the \hyperref[Test1]{Test 1}, the MPDeC methods exhibit superior performance with respect to the MPLM schemes, as shown in Figure \ref{figWPD_MPLM_DeC_Lin_Julia}.
\begin{figure}
\begin{center}
	\scalebox{1}{\includegraphics{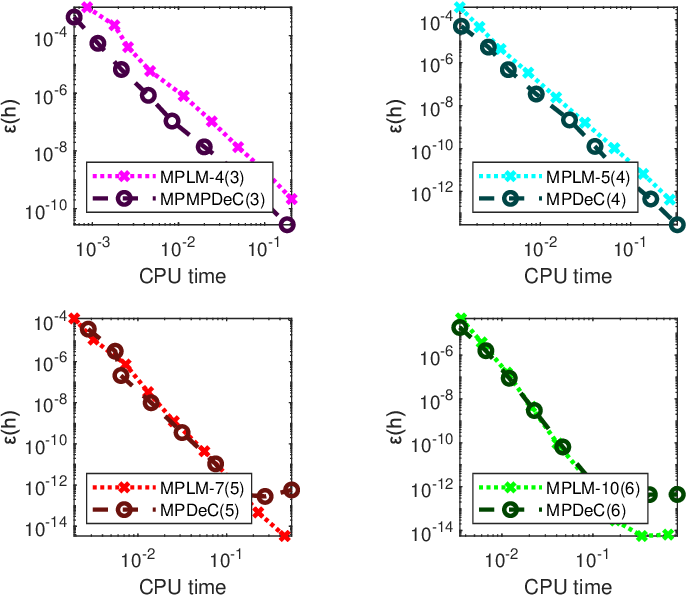}}
\end{center}% \vspace{-0.65cm}
\caption{Work Precision Diagram: Mean Error versus CPU time for the different methods applied to \hyperref[Test1]{Test 1}.  $h=T/2^{5+m}$, $m=1,\ldots,7.$}
\label{figWPD_MPLM_DeC_Lin_Julia}
\end{figure}

% if the computational efforts in obtaining the MPDeC coefficients are disregarded.
% The numerical residual on the linear invariant \eqref{eq:int_primo}, defined as
% \begin{equation}\label{eq:residual_of_the_invariant}
%     R(h)=\max_{0\leq n \leq T/h}|\bm{e}^\mathsf{T}(\bm{y}^0-\bm{y}^n)|,
% \end{equation}
% is analyzed as a function of the stepsize $h$ and reported in Figure \ref{fig:Linear_Residual_DecMPLM} for various values of $p$. The superior performance of the MPLM methods is evident from the work precision diagram of Figure \ref{fig:Linear_WorkPrec_DecMPLM}.

% \begin{figure}
% 	\begin{center}
% 		\scalebox{1}{\includegraphics{Linear_Residual_DecMPLM.eps}}
% 	\end{center}\vspace{-0.65cm}
% 	\caption{Error versus CPU time for the MPLM and MPDeC methods applied to \hyperref[Test1]{Test 1}.} %For comparison purposes Numerical solution obtained by MPLM}
% \label{fig:Linear_Residual_DecMPLM}
% \end{figure}
% \begin{figure}
% 	\begin{center}
% 		\scalebox{1}{\includegraphics{Linear_WorkPrec_DecMPLM.eps}}
% 	\end{center}\vspace{-0.65cm}
% 	\caption{Work precision diagram of the MPLM and MPDeC schemes applied to \hyperref[Test1]{Test 1}.} %For comparison purposes Numerical solution obtained by MPLM}
% \label{fig:Linear_WorkPrec_DecMPLM}
% \end{figure}

\subsection{Test 2: nonlinear test}\label{Test2}
For our second test problem, we consider the non-stiff nonlinear system \cite{MPEuler,Kopecz2018,Kopecz2018second}
\begin{equation}\label{eq:Non_Linear_Test}
			\begin{aligned}
				&y^{\prime}_1(t)=-\dfrac{y_1(t)y_2(t)}{y_1(t)+1}, \\
				&y^{\prime}_2(t)=\phantom{-}\dfrac{y_1(t)y_2(t)}{y_1(t)+1}-ay_2(t), \\
				&y^{\prime}_3(t)=\phantom{-}ay_2(t),
			\end{aligned}\qquad \mbox{with} 
            \quad
            \begin{array}{rr}
			     a=0.3,  \\
			     t\in [0,30], 
			\end{array} 
            \;\ \mbox{and} \;\  
            \bm{y}^0=\begin{pmatrix}9.98 \\ 0.01 \\  0.01 \end{pmatrix},
\end{equation}
obtained from \eqref{eq:PDS} by taking $p_{ij}(\bm{y})=0=d_{ji}(\bm{y})$ for all combination of $i,j=1,\dots,3$ other than
\begin{equation*}
    p_{21}(\bm{y}) =d_{12}(\bm{y})=\dfrac{y_1y_2}{y_1+1}, \qquad\qquad p_{32}(\bm{y}) =d_{23}(\bm{y}) =ay_2.
\end{equation*}
The existence and the uniqueness of a positive solution to \eqref{eq:Non_Linear_Test} comes from \cite[Theorem 1.2]{TORLO2022}. As a matter of fact, given  $0<\bm{y}=(y_1,y_2,y_3)^\mathsf{T}\leq \bm{e}^\mathsf{T}\bm{y}^0,$ 
\begin{equation*}
    \frac{\partial p_{ij}}{\partial y_l}(\bm{y})\leq \max \{a,\bm{e}^\mathsf{T}\bm{y}^0\}, \qquad \mbox{and} \qquad p_{ij}(\bm{y})\overset{\bm{y}\to \bm{0}}{\longrightarrow}0,
    \qquad  i,j,l\in \{1,2,3\}.
\end{equation*}

The PDS \eqref{eq:Non_Linear_Test} models an algal bloom and might be interpreted as a geobiochemical model for the upper oceanic layer in spring, when nutrient rich surface water is captured in the euphotic zone. During the process, nutrients at time $t,$ $y_1(t),$ are taken up by the phytoplankton $y_2(t)$ according to a Michaelis–Menten formulation. Concurrently, the phytoplankton biomass is converted to detritus $y_3(t),$ with a loss fixed rate $a>0$, due the effects of mortality and zooplankton grazing. Thus the system is considered closed and the total biomass remains constant accordingly to the conservation law \eqref{eq:int_primo}. 
\begin{figure}
	\begin{center}
		\scalebox{0.85}{\includegraphics{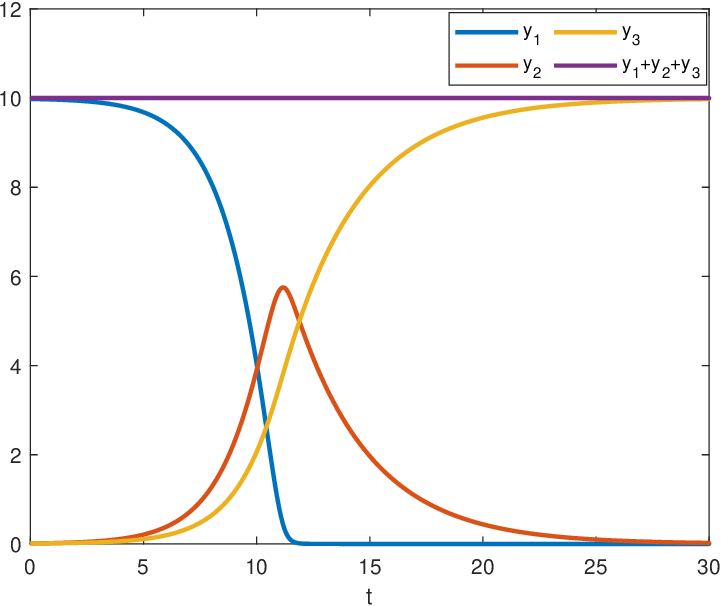}}
	\end{center}% \vspace{-0.65cm}
	\caption{Numerical solution of \hyperref[Test2]{Test 2} by MPLM-$10(6)$ with $h=1.88\cdot 2^{-5}.$} 
\label{fig:Non_Linear_Test_PL}
\end{figure}

The outcomes of the numerical integration of \eqref{eq:Non_Linear_Test} by MPLM-$10(6)$ with $h=1.88\cdot 2^{-5}$ are shown in Figure \ref{fig:Non_Linear_Test_PL}.
The experimental results of Table \ref{tab:Non_Linear_Test_Err} comply with the theoretical findings for this test as well and, from Figure \ref{fig:Non_linear_WPD}, it is clear that for $p\geq 3$ the MPLM-$k(p)$ methods outperform the MPRK3 discretization. For the sake of brevity, here we omit the plots comparing our methods with the MPDeC schemes, which exhibit superior performances also in this test.

% Furthermore, Figure \ref{fig:Non_Linear_Residual_DecMPLM} and Figure \ref{fig:Non_Linear_WorkPrec_DecMPLM} provide a comparison between the MPLM and the MPDeC schemes, highlighting the advantages of the former in terms of conservativity and computational efficiency.
%
\begin{figure}
	\begin{center}
		\scalebox{1}{\includegraphics{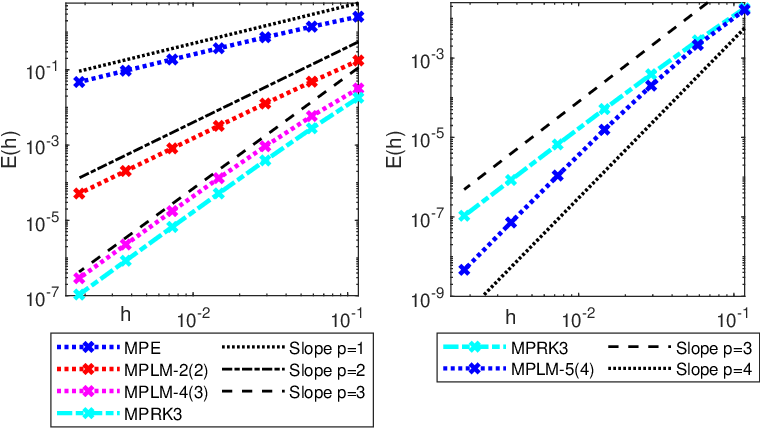}}
	\end{center}%\vspace{-0.65cm}
	\caption{Experimental order for MPE, MPLM-$2(2),$ MPLM-$4(3),$ MPLM-$5(4)$ and MPRK3 applied to \hyperref[Test2]{Test 2}.} %For comparison purposes Numerical solution obtained by MPLM}
\label{fig:Non_Linear_Test_1}
\end{figure}

\begin{figure}
	\begin{center}
		\scalebox{1}{\includegraphics{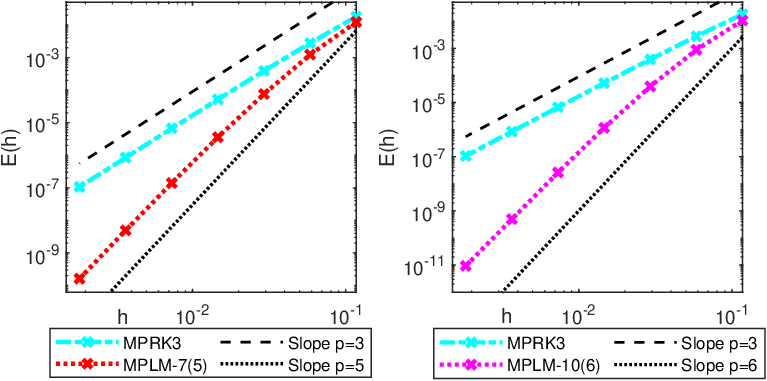}}
	\end{center}%\vspace{-0.65cm}
	\caption{Experimental order for MPLM-$7(5),$ MPLM-$10(6)$ and MPRK3 applied to \hyperref[Test2]{Test 2}.} %For comparison purposes Numerical solution obtained by MPLM}
\label{fig:Non_Linear_Test_2}
\end{figure}
\begin{figure}
\begin{center}
	\scalebox{1}{\includegraphics{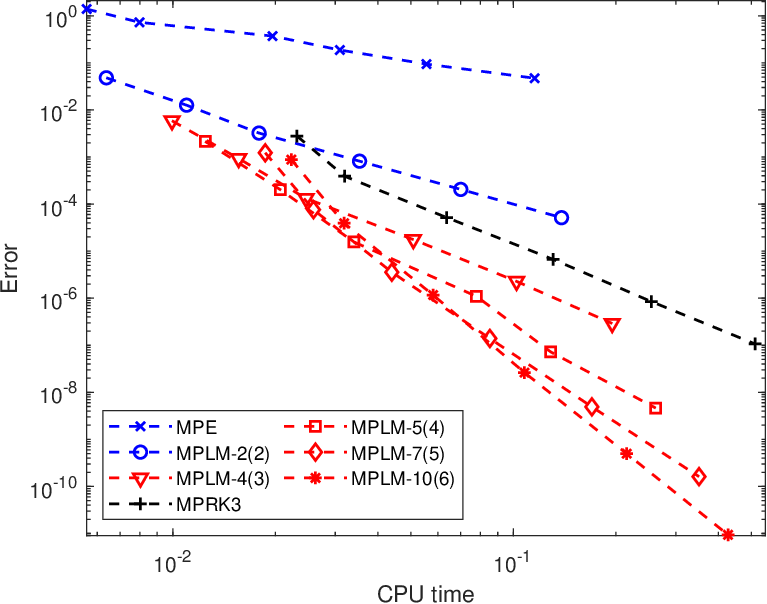}}
\end{center}%\vspace{-0.65cm}
\caption{Work Precision Diagram: Error versus CPU time for the different methods applied to \hyperref[Test2]{Test 2}. $h=T/2^{8+m}$, $m=1,\ldots,6.$}
\label{fig:Non_linear_WPD}
\end{figure}
\begin{table} \center \footnotesize
% table caption is above the table
\caption{Experimental convergence of the numerical solutions to \hyperref[Test2]{Test 2}.}
\label{tab:Non_Linear_Test_Err}       % Give a unique label
\smallskip
% For LaTeX tables use
\begin{tabular}{ccc|ccc|ccc} 
\hline\noalign{\smallskip}
\multicolumn{3}{c|}{MPE} & \multicolumn{3}{|c}{MPLM-$2(2)$} & \multicolumn{3}{|c}{MPLM-$4(3)$}  \\
\noalign{\smallskip}\hline\noalign{\smallskip}
$1.07\cdot h$ & $E(h)$ & $\hat{p}$ & $1.07\cdot h$ & $E(h)$ & $\hat{p}$ & $1.07\cdot h$ & $E(h)$ & $\hat{p}$ \\
$2^{-3}$ & $2.57 \cdot 10^{0}$ & --- & $2^{-3}$ & $1.76 \cdot 10^{-1}$ & --- & $2^{-3}$ & $3.17 \cdot 10^{-2}$ & --- \\
$2^{-4}$ & $1.40 \cdot 10^{0}$ & $0.88$ & $2^{-4}$ & $4.83 \cdot 10^{-2}$ & $1.87$ & $2^{-4}$ & $5.88 \cdot 10^{-3}$ & $2.43$ \\
$2^{-5}$ & $7.28 \cdot 10^{-1}$ & $0.94$ & $2^{-5}$ & $1.26 \cdot 10^{-2}$ & $1.93$ & $2^{-5}$ & $9.29 \cdot 10^{-4}$ & $2.66$ \\
$2^{-6}$ & $3.71 \cdot 10^{-1}$ & $0.97$ & $2^{-6}$ & $3.23 \cdot 10^{-3}$ & $1.97$ & $2^{-6}$ & $1.32 \cdot 10^{-4}$ & $2.82$ \\
$2^{-7}$ & $1.88 \cdot 10^{-1}$ & $0.99$ & $2^{-7}$ & $8.16 \cdot 10^{-4}$ & $1.98$ & $2^{-7}$ & $1.76 \cdot 10^{-5}$ & $2.90$ \\
$2^{-8}$ & $9.43 \cdot 10^{-2}$ & $0.99$ & $2^{-8}$ & $2.05 \cdot 10^{-4}$ & $1.99$ & $2^{-8}$ & $2.28 \cdot 10^{-6}$ & $2.95$ \\
$2^{-9}$ & $4.73 \cdot 10^{-2}$ & $1.00$ & $2^{-9}$ & $5.14 \cdot 10^{-5}$ & $2.00$ & $2^{-9}$ & $2.90 \cdot 10^{-7}$ & $2.97$ \\
\noalign{\smallskip} \hline

\noalign{\smallskip}
\multicolumn{3}{c}{MPLM-$5(4)$} & \multicolumn{3}{|c}{MPLM-$7(5)$} & \multicolumn{3}{|c}{MPLM-$10(6)$} \\
\noalign{\smallskip}\hline\noalign{\smallskip}
$1.07\cdot h$ & $E(h)$ & $\hat{p}$ & $1.07\cdot h$ & $E(h)$ & $\hat{p}$ & $1.07\cdot h$ & $E(h)$ & $\hat{p}$ \\
$2^{-3}$ & $1.64 \cdot 10^{-2}$ & --- & $2^{-3}$ & $1.24 \cdot 10^{-2}$ & --- & $2^{-3}$ & $1.06 \cdot 10^{-2}$ & --- \\
$2^{-4}$ & $2.14 \cdot 10^{-3}$ & $2.94$ & $2^{-4}$ & $1.23 \cdot 10^{-3}$ & $3.34$ & $2^{-4}$ & $8.81 \cdot 10^{-4}$ & $3.59$ \\
$2^{-5}$ & $2.02 \cdot 10^{-4}$ & $3.40$ & $2^{-5}$ & $7.60 \cdot 10^{-5}$ & $4.02$ & $2^{-5}$ & $3.92 \cdot 10^{-5}$ & $4.49$ \\
$2^{-6}$ & $1.57 \cdot 10^{-5}$ & $3.69$ & $2^{-6}$ & $3.57 \cdot 10^{-6}$ & $4.41$ & $2^{-6}$ & $1.17 \cdot 10^{-6}$ & $5.07$ \\
$2^{-7}$ & $1.10 \cdot 10^{-6}$ & $3.84$ & $2^{-7}$ & $1.39 \cdot 10^{-7}$ & $4.68$ & $2^{-7}$ & $2.63 \cdot 10^{-8}$ & $5.48$ \\
$2^{-8}$ & $7.23 \cdot 10^{-8}$ & $3.92$ & $2^{-8}$ & $4.91 \cdot 10^{-9}$ & $4.83$ & $2^{-8}$ & $4.97 \cdot 10^{-10}$ & $5.72$ \\
$2^{-9}$ & $4.64 \cdot 10^{-9}$ & $3.96$ & $2^{-9}$ & $1.62 \cdot 10^{-10}$ & $4.92$ & $2^{-9}$ & $9.40 \cdot 10^{-12}$ & $5.73$ \\
\noalign{\smallskip} \hline

\end{tabular}
\end{table}

% \begin{figure}
% 	\begin{center}
% 		\scalebox{1}{\includegraphics{Non_Linear_Residual_DecMPLM.eps}}
% 	\end{center}\vspace{-0.65cm}
% 	\caption{Error versus CPU time for the MPLM and MPDeC methods applied to \hyperref[Test2]{Test 2}.} 
% \label{fig:Non_Linear_Residual_DecMPLM}
% \end{figure}
% \begin{figure}
% 	\begin{center}
% 		\scalebox{1}{\includegraphics{Non_Linear_WorkPrec_DecMPLM.eps}}
% 	\end{center}\vspace{-0.65cm}
% 	\caption{Work precision diagram of the MPLM and MPDeC schemes applied to \hyperref[Test2]{Test 2}.} 
% \label{fig:Non_Linear_WorkPrec_DecMPLM}
% \end{figure}

\subsection{Test 3: Brusselator test}\label{Test3}
 Our next experiment addresses a typical nonlinear chemical kinetics problem modeled by the original Brusselator system\cite{Bruss,Hairer,Bonaventura2017}  
 \begin{equation}\label{eq:Bruss}
			\begin{aligned}
				& y^{\prime}_1(t)=-k_1 y_1(t), \\
                & y^{\prime}_2(t)=-k_2 y_2(t)y_5(t), \\ 
                & y^{\prime}_3(t)=k_2 y_2(t)y_5(t), \\
				& y^{\prime}_4(t)=k_4y_5(t), \\ 
                & y^{\prime}_5(t)=k_1y_1(t)-k_2y_2(t)y_5(t)+k_3y^2_5(t)y_6(t)-k_4y_5(t), \\
				& y^{\prime}_6(t)=k_2y_2(t)y_5(t)-k_3y^2_5(t)y_6(t), 
			\end{aligned}\qquad \mbox{with} 
            \quad
            \bm{y}^0=\begin{pmatrix}10 \\ 10 \\  0 \\ 0 \\ 0.1 \\ 0.1 \end{pmatrix}.
 \end{equation}
 The differential system \eqref{eq:Bruss} falls in the form \eqref{eq:PDS}, setting
 \begin{equation*}
        \begin{split}
            &p_{32}(\bm{y}) = d_{23}(\bm{y}) = k_2 y_2 y_5, \qquad p_{45}(\bm{y}) = d_{54}(\bm{y}) = k_4 y_5, \qquad p_{51}(\bm{y}) = d_{15}(\bm{y}) = k_1 y_1, \\
            & p_{56}(\bm{y}) = d_{65}(\bm{y}) = k_3 y^2_5 y_6, \qquad p_{65}(\bm{y}) = d_{56}(\bm{y}) = k_2 y_2 y_5,
        \end{split}
 \end{equation*}
 and $p_{ij}(\bm{y}) = d_{ji}(\bm{y}) = 0$ for all other combinations of $i$ and $j$ in $[1,6]\cap \mathbb{N}$. Resorting to \cite[Theorem 1.2]{TORLO2022} the existence of a unique non-negative solution to \eqref{eq:Bruss} can be proved.

  \begin{figure}
	\begin{center}
		\scalebox{0.85}{\includegraphics{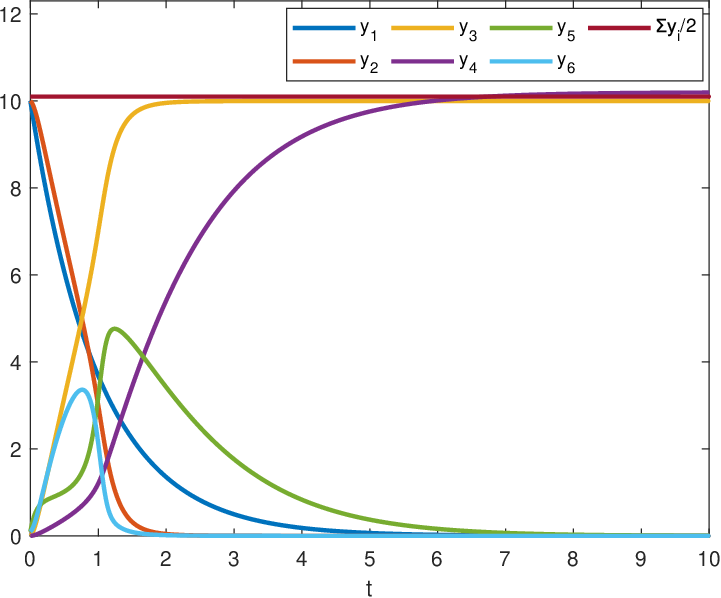}}
	\end{center}% \vspace{-0.65cm}
	\caption{Numerical solution of \hyperref[Test3]{Test 3} by MPLM-$10(6)$ with $h=1.25\cdot 2^{-5}.$} %For comparison purposes Numerical solution obtained by MPLM}
\label{fig:Bruss_Test_PL}
\end{figure}
For this test problem the $\sigma$-embedding technique \eqref{eq:sigma_embedding} is not suitable for use due the presence of zero components in the initial value $\bm{y}^0.$ As already pointed out, we set $y^0_3=y^0_4=\mbox{\texttt{realmin}}\approx2.26\cdot 10^{-308}.$
The simulation of the system \eqref{eq:Bruss} by the MPLM-$10(6)$ method with $h=1.25\cdot 2^{-5}$, assuming $t\in [0,10]$ and $k_l=1,$ $l=1,\dots,4,$ results in the numerical solution of Figure \ref{fig:Bruss_Test_PL}. Accordingly to the theoretical findings of Lemmas \ref{lem:positivity} and \ref{lem:conservativity}, both the positivity and the conservativity are guaranteed. Furthermore, the results of Table \ref{tab:Bruss_Test_Err} and Figures \ref{fig:Bruss_Test_1} and \ref{fig:Bruss_Test_2} underline the effectiveness of the choice to slightly variate the initial state. The work precision diagram of Figure \ref{fig:Bruss_Test_WPD} confirms that, assuming the same mean execution time, the MPLM-$k(p)$ methods for $p\geq 3$ attain higher accuracy than the MPRK3 scheme.
\begin{figure}
	\begin{center}
		\scalebox{0.95}{\includegraphics{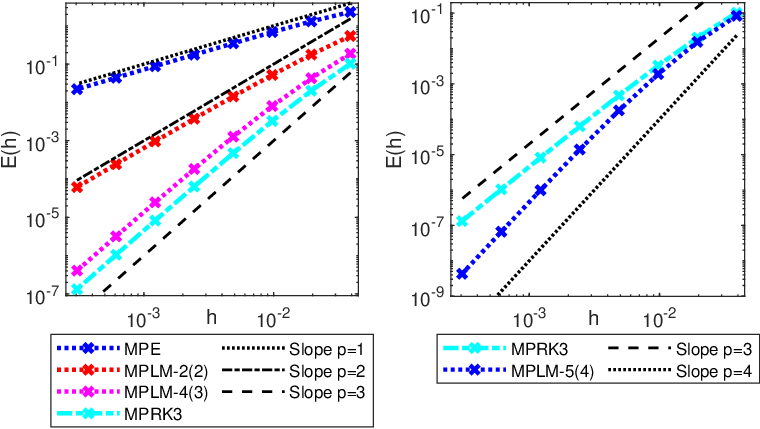}}
	\end{center}\vspace{-0.3cm}
	\caption{Experimental order for MPE, MPLM-$2(2),$ MPLM-$4(3),$ MPLM-$5(4)$ and MPRK3 applied to \hyperref[Test3]{Test 3}.} %For comparison purposes Numerical solution obtained by MPLM}
\label{fig:Bruss_Test_1}
\end{figure}
\begin{figure}
	\begin{center}
		\scalebox{0.95}{\includegraphics{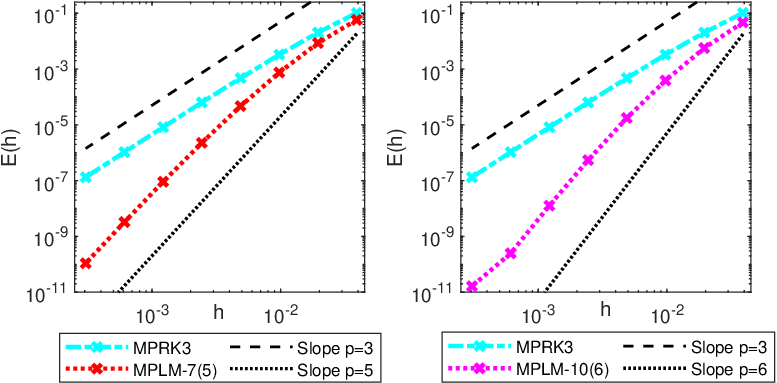}}
	\end{center}\vspace{-0.3cm}
	\caption{Experimental order for MPLM-$7(5),$ MPLM-$10(6)$ and MPRK3 applied to \hyperref[Test3]{Test 3}.} %For comparison purposes Numerical solution obtained by MPLM}
\label{fig:Bruss_Test_2}
\end{figure}
\begin{figure}
\begin{center}
	\scalebox{1}{\includegraphics{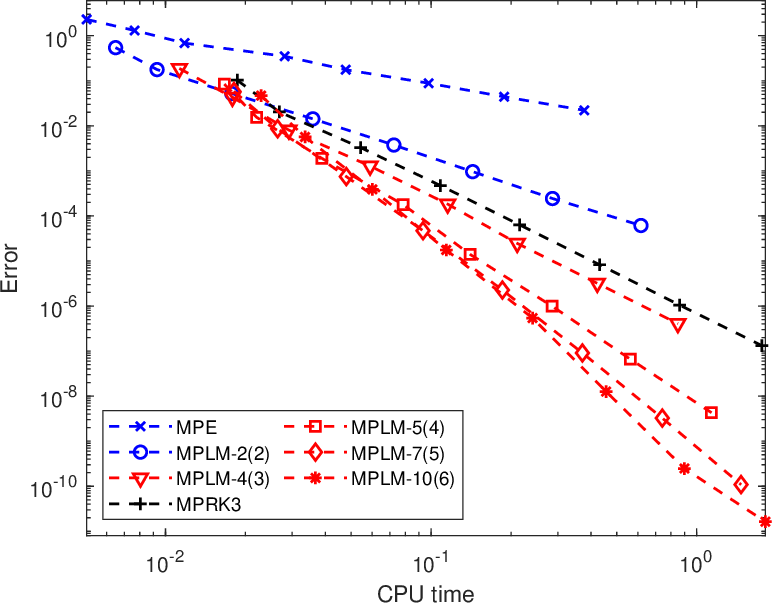}}
\end{center}%\vspace{-0.65cm}
\caption{Work Precision Diagram: Error versus CPU time for the different methods applied to \hyperref[Test3]{Test 3}. $h=T/2^{8+m}$, $m=0,\ldots,7.$}
\label{fig:Bruss_Test_WPD}
\end{figure}
\begin{table} \center \footnotesize
% table caption is above the table
\caption{Experimental convergence of the numerical solutions to \hyperref[Test3]{Test 3}.}
\label{tab:Bruss_Test_Err}       % Give a unique label
\smallskip
% For LaTeX tables use
\begin{tabular}{ccc|ccc|ccc}

\hline\noalign{\smallskip}
\multicolumn{3}{c|}{MPE} & \multicolumn{3}{|c}{MPLM-$2(2)$} & \multicolumn{3}{|c}{MPLM-$4(3)$}  \\
\noalign{\smallskip}\hline\noalign{\smallskip}
$0.80\cdot h$ & $E(h)$ & $\hat{p}$ & $0.80\cdot h$ & $E(h)$ & $\hat{p}$ & $0.80\cdot h$ & $E(h)$ & $\hat{p}$ \\
$2^{-5}$ & $2.30 \cdot 10^{0}$ & --- & $2^{-5}$ & $5.44 \cdot 10^{-1}$ & --- & $2^{-5}$ & $1.87 \cdot 10^{-1}$ & --- \\
$2^{-6}$ & $1.31 \cdot 10^{0}$ & $0.82$ & $2^{-6}$ & $1.77 \cdot 10^{-1}$ & $1.62$ & $2^{-6}$ & $4.29 \cdot 10^{-2}$ & $2.13$ \\
$2^{-7}$ & $6.86 \cdot 10^{-1}$ & $0.93$ & $2^{-7}$ & $5.21 \cdot 10^{-2}$ & $1.77$ & $2^{-7}$ & $8.04 \cdot 10^{-3}$ & $2.42$ \\
$2^{-8}$ & $3.49 \cdot 10^{-1}$ & $0.97$ & $2^{-8}$ & $1.43 \cdot 10^{-2}$ & $1.87$ & $2^{-8}$ & $1.28 \cdot 10^{-3}$ & $2.65$ \\
$2^{-9}$ & $1.76 \cdot 10^{-1}$ & $0.99$ & $2^{-9}$ & $3.75 \cdot 10^{-3}$ & $1.93$ & $2^{-9}$ & $1.83 \cdot 10^{-4}$ & $2.81$ \\
$2^{-10}$ & $8.82 \cdot 10^{-2}$ & $1.00$ & $2^{-10}$ & $9.62 \cdot 10^{-4}$ & $1.96$ & $2^{-10}$ & $2.46 \cdot 10^{-5}$ & $2.90$ \\
$2^{-11}$ & $4.42 \cdot 10^{-2}$ & $1.00$ & $2^{-11}$ & $2.44 \cdot 10^{-4}$ & $1.98$ & $2^{-11}$ & $3.19 \cdot 10^{-6}$ & $2.95$ \\
$2^{-12}$ & $2.21 \cdot 10^{-4}$ & $1.00$ & $2^{-12}$ & $6.13 \cdot 10^{-5}$ & $1.99$ & $2^{-12}$ & $4.07 \cdot 10^{-7}$ & $2.97$ \\
\noalign{\smallskip} \hline

\noalign{\smallskip}
\multicolumn{3}{c}{MPLM-$5(4)$} & \multicolumn{3}{|c}{MPLM-$7(5)$} & \multicolumn{3}{|c}{MPLM-$10(6)$}  \\
\noalign{\smallskip}\hline\noalign{\smallskip}
$0.80\cdot h$ & $E(h)$ & $\hat{p}$ & $0.80\cdot h$ & $E(h)$ & $\hat{p}$ & $0.80\cdot h$ & $E(h)$ & $\hat{p}$ \\
$2^{-5}$ & $8.40 \cdot 10^{-2}$ & --- & $2^{-5}$ & $5.80 \cdot 10^{-2}$ & --- & $2^{-5}$ & $4.68 \cdot 10^{-2}$ & --- \\
$2^{-6}$ & $1.54 \cdot 10^{-2}$ & $2.44$ & $2^{-6}$ & $8.60 \cdot 10^{-3}$ & $2.75$ & $2^{-6}$ & $5.66 \cdot 10^{-3}$ & $3.05$ \\
$2^{-7}$ & $1.90 \cdot 10^{-3}$ & $3.03$ & $2^{-7}$ & $7.48 \cdot 10^{-4}$ & $3.52$ & $2^{-7}$ & $3.89 \cdot 10^{-4}$ & $3.86$ \\
$2^{-8}$ & $1.78 \cdot 10^{-4}$ & $3.41$ & $2^{-8}$ & $4.70 \cdot 10^{-5}$ & $3.99$ & $2^{-8}$ & $1.75 \cdot 10^{-5}$ & $4.47$ \\
$2^{-9}$ & $1.40 \cdot 10^{-5}$ & $3.67$ & $2^{-9}$ & $2.27 \cdot 10^{-6}$ & $4.37$ & $2^{-9}$ & $5.41 \cdot 10^{-7}$ & $5.02$ \\
$2^{-10}$ & $9.93 \cdot 10^{-7}$ & $3.82$ & $2^{-10}$ & $9.09 \cdot 10^{-8}$ & $4.64$ & $2^{-10}$ & $1.27 \cdot 10^{-8}$ & $5.42$ \\
$2^{-11}$ & $6.63 \cdot 10^{-8}$ & $3.90$ & $2^{-11}$ & $3.26 \cdot 10^{-9}$ & $4.80$ & $2^{-11}$ & $2.48 \cdot 10^{-10}$ & $5.68$ \\
$2^{-12}$ & $4.30 \cdot 10^{-9}$ & $3.95$ & $2^{-12}$ & $1.09 \cdot 10^{-10}$ & $4.90$ & $2^{-12}$ & $1.63 \cdot 10^{-11}$ & $3.92$ \\
\noalign{\smallskip} \hline

\end{tabular}
\end{table}

Figure \ref{fig:Bruss_Test_WPD_Julia} shows, for $p\geq 3,$ a comparison between the MPLM$-k(p)$ and the MPDeC$(p)$ methods in terms of the accuracy-cost trade off, highlighting the advantages of the former in terms of computational efficiency. Here, the mean error is computed by following \eqref{eq:Relative_Error_all_components}. The outcomes of the comparison, which differ from those obtained in \hyperref[Test1]{Test 1} and \hyperref[Test2]{Test 2}, align with the inherent characteristics of each scheme. Specifically, MPDeC methods address multiple linear systems at each step, whereas MPLM methods necessitate an initialization procedure and the calculation of the PWDs. Consequently, for systems with smaller dimensions, the former demonstrate better efficiency. Conversely, as the number of components increases, MPLM methods exhibit reduced computational complexity and become more advantageous.

\begin{figure}
\begin{center}
	\scalebox{1}{\includegraphics{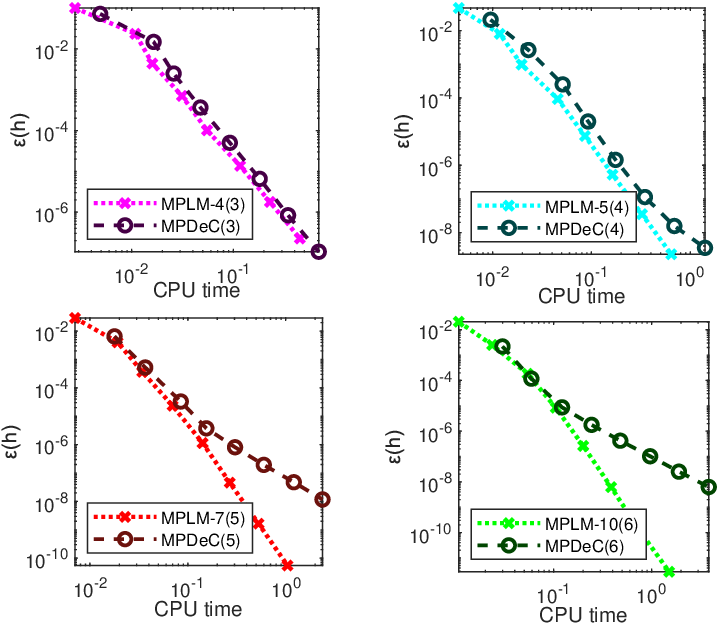}}
\end{center}% \vspace{-0.65cm}
\caption{Work Precision Diagram: Mean Error versus CPU time for the different methods applied to \hyperref[Test3]{Test 3}. $h=2^{-m}$, $m=5,\ldots,12.$}
\label{fig:Bruss_Test_WPD_Julia}
\end{figure}

\subsection{Test 4: SACEIRQD COVID-19 model}\label{Test4}
Our fourth test problem deals with the modified
Susceptible-Infected-Recovered-Dead epidemic model  
\begin{equation}\label{eq:SIRD}
			\begin{aligned}
				& S^{\prime}(t)=-\left(\alpha+\dfrac{\beta I(t)+\sigma A(t)}{N_P}+\eta\right) S(t), \\
                & A^{\prime}(t)=-\tau A(t) +\xi E(t), \\ 
                & C^{\prime}(t)=\alpha S(t)-\mu C(t), \\
				& E^{\prime}(t)=\left(\dfrac{\beta I(t)+\sigma A(t)}{N_P}+\eta\right) S(t)+\mu C(t)-(\gamma + \xi)E(t), \\ 
                & I^{\prime}(t)=\tau A(t)+\gamma E(t) -\delta I(t), \\
				& R^{\prime}(t)=\lambda Q(t), \\
                & Q^{\prime}(t)=\delta I(t)-\lambda Q(t)-k_d Q(t), \\
                & D^{\prime}(t)=k_d Q(t).
			\end{aligned}
 \end{equation}
 
Originally introduced in \cite{Sen2021} to analyze COVID-19 data, it incorporates the effect of asymptomatic infections and the influence of containment, isolation and quarantine measures on the spread of the disease. 
The non-negative states variables of the model \eqref{eq:SIRD} represent the sizes at time $t$ (days) of the eight disjoint compartments in which a closed population of $N_P$ individuals is partitioned: susceptible (S), asymptomatic (A), confined (C), exposed (E), infected (I),  recovered (R), quarantined (Q), dead (D). The absence of migration turnover in the model leads to the conservation law 
\begin{equation*}
    S(t)+A(t)+C(t)+E(t)+I(t)+R(t)+Q(t)+D(t)=N_P, \qquad \; \forall t\geq 0.
\end{equation*}
We refer to \cite{Sen2021} for the details on the physical interpretation of the positive constants $\alpha$, $\beta$, $\gamma$, $\delta$, $\sigma$, $\eta$, $\tau$, $\xi$, $\mu$, $\lambda$ and $k_d$. 

In order to reformulate the system \eqref{eq:SIRD} as a PDS, we introduce the function \\ $\bm{y}(t)=(S(t),A(t),C(t),E(t),I(t),R(t),Q(t),D(t))^\mathsf{T}$ and the production-destruction terms as follows
\begin{equation*}
    \begin{split}
        & p_{24}(\bm{y})=d_{42}(\bm{y})=\xi y_4, \qquad\qquad\qquad\qquad\quad\;\
        p_{31}(\bm{y})=d_{13}(\bm{y})=\alpha y_1, \\
        & p_{41}(\bm{y})=d_{14}(\bm{y})=y_1\left(\eta+\frac{\beta y_5+\sigma y_2}{N_P}\right), \qquad 
        p_{43}(\bm{y})=d_{34}(\bm{y})=\mu y_3, \\
        & p_{52}(\bm{y})=d_{25}(\bm{y})=\tau y_2, \qquad\qquad\qquad\qquad\quad\;\
        p_{54}(\bm{y})=d_{45}(\bm{y})=\gamma y_4, \\
        & p_{67}(\bm{y})=d_{76}(\bm{y})=\lambda y_7, \qquad\qquad\qquad\qquad\quad\;\
        p_{75}(\bm{y})=d_{57}(\bm{y})=\delta y_5, \\
        & p_{87}(\bm{y})=d_{78}(\bm{y})=K_d y_7
    \end{split}
\end{equation*}
and $p_{ij}(\bm{y}) = d_{ji}(\bm{y}) = 0$ for all other combinations of $i,j \in [1,8]\cap \mathbb{N}$.

In \cite[Table 1]{Sen2021} the parameters of the model were fitted to the COVID-19 time series datasets of infected, recovered and death cases for different countries. Here, we adopt for the parameters of the model the values provided by the authors for Italy, i.e.
\begin{equation}\label{eq:SIRD_Parameters}
    \begin{split}
        & N_P=6.046 \cdot 10^7, \qquad \!\!\!\!\!\!
        \alpha=0.0194, \qquad \quad \;\
        \!\beta=7.567, \qquad\quad \;\ \; \; \mu=2.278 \cdot 10^{-6}, \\
        & \eta=9.180 \cdot 10^{-7}, \ \qquad \!\!\!\!\!\!
        \sigma=1.4633 \cdot 10^{-3}, \quad \!\tau=1.109 \cdot 10^{-4}, \quad \;
        \; \xi=0.263, \\
        &  \gamma=0.021, \qquad \qquad \quad \!\!\!\!\!\!
        \delta=0.077, \qquad \qquad \!\lambda_0=0.157, \qquad \quad \: \;\; \lambda_1=0.025,  \\
        & k_{d0}=0.779, \qquad \quad k_{d1}=0.061, \qquad \quad \;  \bm{y}^0=(60459997,0,0,1,1,0,1,0)^\mathsf{T}.
    \end{split}
\end{equation}
Furthermore, we set $\lambda=10^{-4}\lambda_0 \int_0^{10^4} e^{-\lambda_1 t} dt$ and $k_d=10^{-4}k_{d0} \int_0^{10^4} e^{-k_{d1} t} dt.$

  \begin{figure}
	\begin{center}
		\scalebox{1.0}{\includegraphics{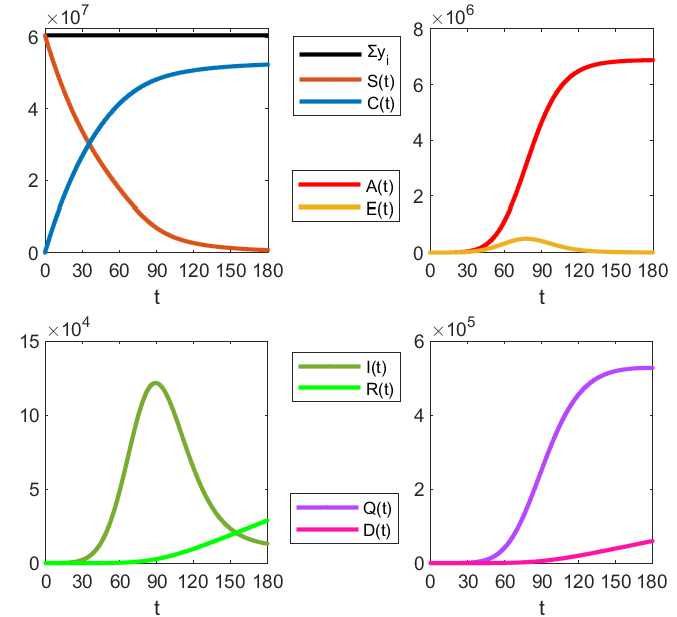}}
	\end{center}% \vspace{-0.65cm}
	\caption{Numerical solution of \hyperref[Test4]{Test 4} by MPLM-$10(6)$ with $h=1.4\cdot 2^{-1}.$ } %For comparison purposes Numerical solution obtained by MPLM}
\label{fig:SIRD_Test_PL}
\end{figure} 

The mathematical statements concerning the positivity (Lemma \ref{lem:positivity}) and conservativity (Lemma \ref{lem:conservativity}) of the MPLM methods are confirmed in Figure \ref{fig:SIRD_Test_PL}, where the numerical simulation of \eqref{eq:SIRD} by the MPLM-$10(6)$ for $t\in [0,180]$ is reported. For this test, to avoid the issues of null components in the initial value, we set $y_i^0=\texttt{realmin}\approx2.26\cdot 10^{-308}$ for $i=2,3,6,8.$ Furthermore, due to the presence of large values in the components of the solution, we consider the relative maximum error 
\begin{equation*}
    e(h)=\dfrac{E(h)}{\max_{0\leq n \leq T/h}\|\bm{y}^{n(ref)}\|_\infty},
\end{equation*}
with $E(h)$ defined in \eqref{eq:E(h)}. From Table \ref{tab:SIRD_Test_Err} and Figures \ref{fig:SIRD_Test_1} and \ref{fig:SIRD_Test_2} it is clear that the numerical solution behaves in accordance with the theoretical results and the experimental order of convergence coincides with the expected one. Moreover, the work precision diagram of Figure \ref{fig:SIRD_Test_WPD} confirms, also for this test, the trends observed in \hyperref[Test3]{Test 3}.
% shows a work precision diagram comparing the mean execution time over ten runs against the accuracy achieved by the different methods. It is clear from it that the MPLM-$k(p)$ methods are, for $p\geq 3,$ more efficient than the MPRK3 scheme. 

%
\begin{figure}
	\begin{center}
		\scalebox{1}{\includegraphics{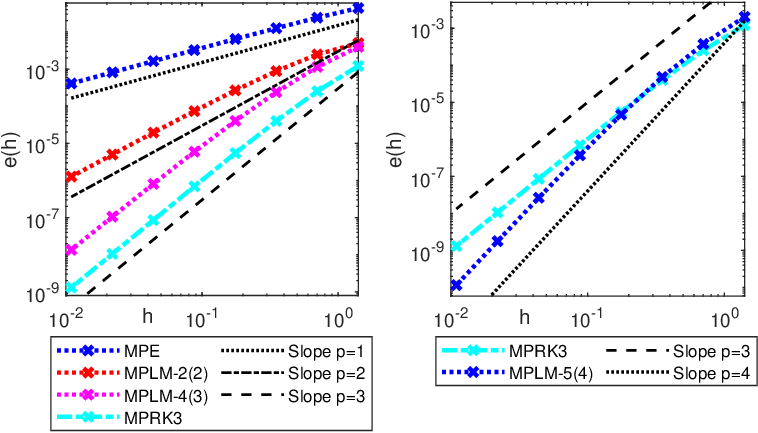}}
	\end{center}%\vspace{-0.65cm}
	\caption{Experimental order for MPE, MPLM-$2(2),$ MPLM-$4(3),$ MPLM-$5(4)$ and MPRK3 applied to \hyperref[Test4]{Test 4}.} %For comparison purposes Numerical solution obtained by MPLM}
\label{fig:SIRD_Test_1}
\end{figure}
\begin{figure}
	\begin{center}
		\scalebox{1}{\includegraphics{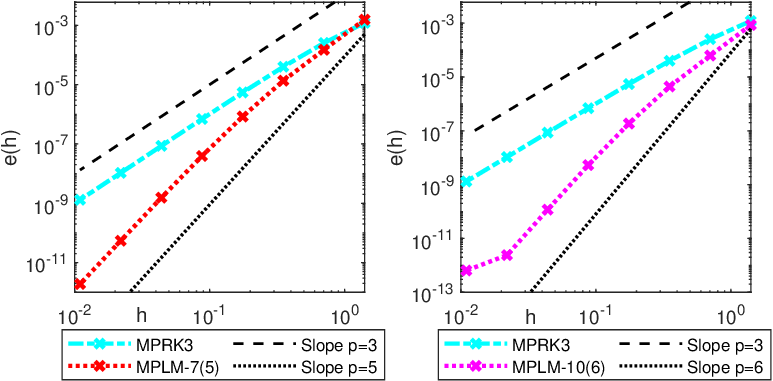}}
	\end{center}%\vspace{-0.65cm}
	\caption{Experimental order for MPLM-$7(5),$ MPLM-$10(6)$ and MPRK3 applied to \hyperref[Test4]{Test 4}.} %For comparison purposes Numerical solution obtained by MPLM}
\label{fig:SIRD_Test_2}
\end{figure}
\begin{figure}
\begin{center}
	\scalebox{1}{\includegraphics{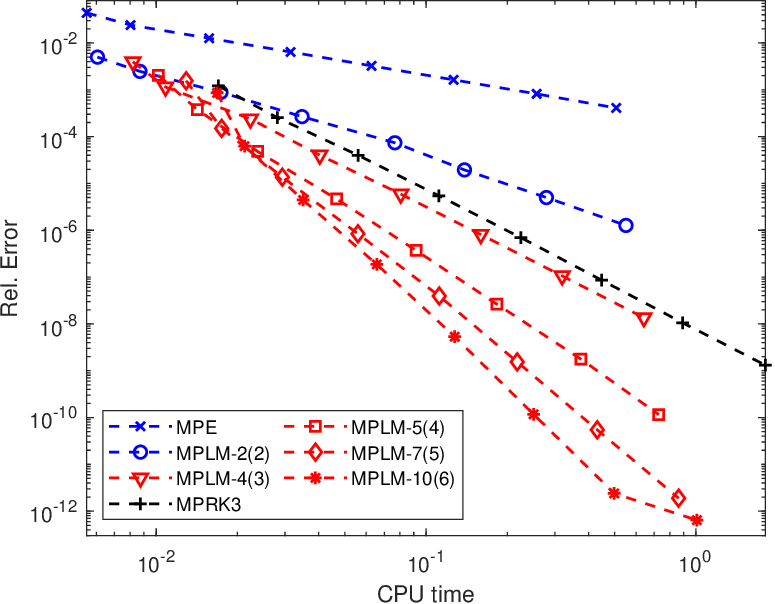}}
\end{center}%\vspace{-0.65cm}
\caption{Work Precision Diagram: Error versus CPU time for the different methods applied to \hyperref[Test4]{Test 4}. $h=T/2^{7+m}$, $m=0,\ldots,7.$}
\label{fig:SIRD_Test_WPD}
\end{figure}
\begin{table} \center \footnotesize
% table caption is above the table
\caption{Experimental convergence of the numerical solutions to \hyperref[Test4]{Test 4}.}
\label{tab:SIRD_Test_Err}       % Give a unique label
\smallskip
% For LaTeX tables use
\begin{tabular}{ccc|ccc|ccc}

\hline\noalign{\smallskip}
\multicolumn{3}{c}{MPE} & \multicolumn{3}{|c}{MPLM-$2(2)$} & \multicolumn{3}{|c}{MPLM-$4(3)$}  \\
\noalign{\smallskip}\hline\noalign{\smallskip}
$0.35\cdot h$ & $e(h)$ & $\hat{p}$ & $0.35\cdot h$ & $e(h)$ & $\hat{p}$ & $0.35\cdot h$ & $e(h)$ & $\hat{p}$ \\
$2^{-1}$ & $4.39 \cdot 10^{-2}$ & --- & $2^{-1}$ & $4.98 \cdot 10^{-3}$ & --- & $2^{-1}$ & $3.96 \cdot 10^{-3}$ & --- \\
$2^{-2}$ & $2.41 \cdot 10^{-2}$ & $0.87$ & $2^{-2}$ & $2.47 \cdot 10^{-3}$ & $1.01$ & $2^{-2}$ & $1.14 \cdot 10^{-3}$ & $1.79$ \\
$2^{-3}$ & $1.26 \cdot 10^{-2}$ & $0.94$ & $2^{-3}$ & $8.82 \cdot 10^{-4}$ & $1.49$ & $2^{-3}$ & $2.38 \cdot 10^{-4}$ & $2.26$ \\
$2^{-4}$ & $6.42 \cdot 10^{-3}$ & $0.97$ & $2^{-4}$ & $2.67 \cdot 10^{-4}$ & $1.72$ & $2^{-4}$ & $4.02 \cdot 10^{-5}$ & $2.56$ \\
$2^{-5}$ & $3.24 \cdot 10^{-3}$ & $0.99$ & $2^{-5}$ & $7.38 \cdot 10^{-5}$ & $1.85$ & $2^{-5}$ & $5.92 \cdot 10^{-6}$ & $2.76$ \\
$2^{-6}$ & $1.63 \cdot 10^{-3}$ & $0.99$ & $2^{-6}$ & $1.94 \cdot 10^{-5}$ & $1.92$ & $2^{-6}$ & $8.09 \cdot 10^{-7}$ & $2.87$ \\
$2^{-7}$ & $8.17 \cdot 10^{-4}$ & $1.00$ & $2^{-7}$ & $4.99 \cdot 10^{-6}$ & $1.96$ & $2^{-7}$ & $1.06 \cdot 10^{-7}$ & $2.93$ \\
$2^{-8}$ & $4.09 \cdot 10^{-4}$ & $1.00$ & $2^{-8}$ & $1.27 \cdot 10^{-6}$ & $1.98$ & $2^{-8}$ & $1.36 \cdot 10^{-8}$ & $2.97$ \\
\noalign{\smallskip} \hline

\noalign{\smallskip}
\multicolumn{3}{c}{MPLM-$5(4)$} & \multicolumn{3}{|c}{MPLM-$7(5)$} & \multicolumn{3}{|c}{MPLM-$10(6)$}  \\
\noalign{\smallskip}\hline\noalign{\smallskip}
$0.35\cdot h$ & $e(h)$ & $\hat{p}$ & $0.35\cdot h$ & $e(h)$ & $\hat{p}$ & $0.35\cdot h$ & $e(h)$ & $\hat{p}$ \\
$2^{-1}$ & $2.03 \cdot 10^{-3}$ & --- & $2^{-1}$ & $1.55 \cdot 10^{-3}$ & --- & $2^{-1}$ & $8.56 \cdot 10^{-4}$ & --- \\
$2^{-2}$ & $3.76 \cdot 10^{-4}$ & $2.43$ & $2^{-2}$ & $1.50 \cdot 10^{-4}$ & $3.38$ & $2^{-2}$ & $6.25 \cdot 10^{-5}$ & $3.78$ \\
$2^{-3}$ & $4.83 \cdot 10^{-5}$ & $2.96$ & $2^{-3}$ & $1.36 \cdot 10^{-5}$ & $3.46$ & $2^{-3}$ & $4.43 \cdot 10^{-6}$ & $3.82$ \\
$2^{-4}$ & $4.65 \cdot 10^{-6}$ & $3.38$ & $2^{-4}$ & $8.41 \cdot 10^{-7}$ & $4.02$ & $2^{-4}$ & $1.86 \cdot 10^{-7}$ & $4.57$ \\
$2^{-5}$ & $3.71 \cdot 10^{-7}$ & $3.65$ & $2^{-5}$ & $3.94 \cdot 10^{-8}$ & $4.42$ & $2^{-5}$ & $5.32 \cdot 10^{-9}$ & $5.13$ \\
$2^{-6}$ & $2.64 \cdot 10^{-8}$ & $3.81$ & $2^{-6}$ & $1.54 \cdot 10^{-9}$ & $4.67$ & $2^{-6}$ & $1.17 \cdot 10^{-10}$ & $5.50$ \\
$2^{-7}$ & $1.78 \cdot 10^{-9}$ & $3.90$ & $2^{-7}$ & $5.46 \cdot 10^{-11}$ & $4.82$ & $2^{-7}$ & $2.39 \cdot 10^{-12}$ & $5.62$ \\
$2^{-8}$ & $1.16 \cdot 10^{-10}$ & $3.94$ & $2^{-8}$ & $1.89 \cdot 10^{-12}$ & $4.85$ & $2^{-8}$ & $6.43 \cdot 10^{-13}$ & $1.90$ \\
\noalign{\smallskip} \hline

\end{tabular}
\end{table}

The efficient integration of \eqref{eq:SIRD}-\eqref{eq:SIRD_Parameters} using the MPDeC codes in \cite{Torlo_Rep} is not feasible and results in a numerical solution exhibiting several \texttt{NAN}. To overcome this issue, we consider the system \eqref{eq:SIRD} with the same parameters in \eqref{eq:SIRD_Parameters} but with modified initial values 
\begin{equation}\label{eq: SIRD_Initial_Cond_2}
    \bm{\bar{y}}^0=(60459997,10^{-10},10^{-10},1,1,10^{-10},1,10^{-10})^\mathsf{T}, \qquad \mbox{and} \qquad \bm{\bar{\bar{y}}}^0=10^4  \bm{e}.
\end{equation} 
The work precision diagrams associated with the MPLM and MPDeC simulations of \hyperref[Test4]{Test 4}, initialized with $\bm{y}(0)=\bm{\bar{y}}^0$ and $\bm{y}(0)=\bm{\bar{\bar{y}}}^0$, are depicted in Figures \ref{fig:WPD_MPLM_DeC_SIRD_10000} and \ref{fig:WPD_MPLM_DeC_SIRD_men10}, respectively. The plotted data indicate that, at least for $p>3,$ MPLM methods require fewer computational resources than MPDeC schemes to achieve a predefined level of accuracy. On the basis of the comparative analysis we performed, the MPLM integrators emerge as more suitable for accurately simulating real-world scenarios of the infectious disease outbreak model \eqref{eq:SIRD}. Furthermore, even in the unrealistic case of $\bm{y}(0)=\bm{\bar{\bar{y}}}^0$, they demonstrate superior efficiency compared to MPDeC methods.

\begin{figure}
\begin{center}
	\scalebox{1}{\includegraphics{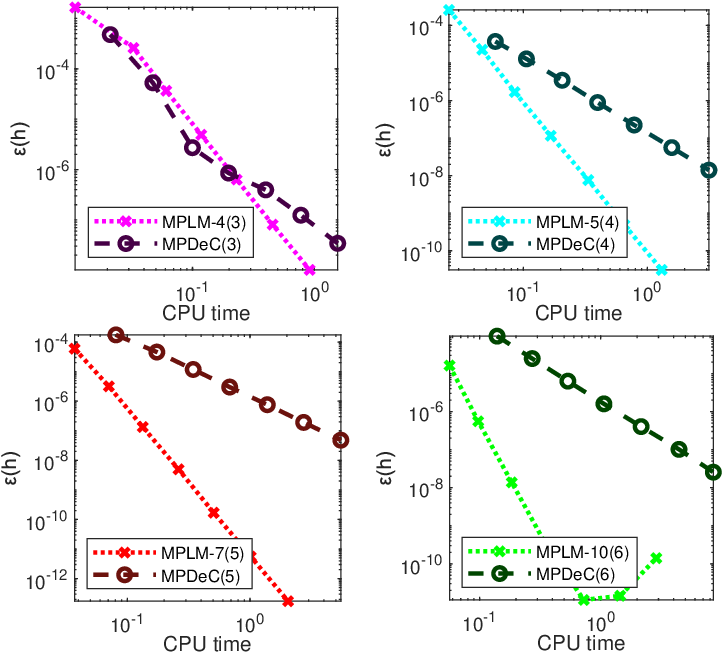}}
\end{center}%\vspace{-0.65cm}
\caption{Work Precision Diagram: Mean Error versus CPU time for the different methods applied to \hyperref[Test4]{Test 4} with $\bm{y}(0)=\bm{\bar{y}}^0$ in \eqref{eq: SIRD_Initial_Cond_2}. $h=2^{-m}$, $m=2,\ldots,8.$}
\label{fig:WPD_MPLM_DeC_SIRD_10000}
\end{figure}
\begin{figure}
\begin{center}
	\scalebox{1}{\includegraphics{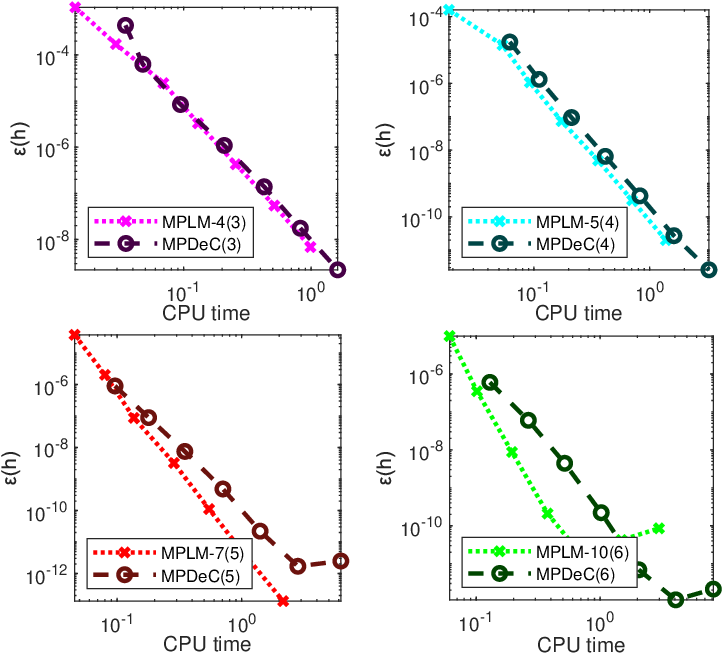}}
\end{center}%\vspace{-0.65cm}
\caption{Work Precision Diagram: Mean Error versus CPU time for the different methods applied to \hyperref[Test4]{Test 4} with $\bm{y}(0)=\bm{\bar{\bar{y}}}^0$ in \eqref{eq: SIRD_Initial_Cond_2}. $h=2^{-m}$, $m=2,\ldots,8.$}
\label{fig:WPD_MPLM_DeC_SIRD_men10}
\end{figure}

\subsection{Test 5: spatially heterogeneous diffusion equation}\label{Test5}
With the aim of assessing the performance of MPLM schemes on larger-scale problems, we turn our attention to the following Partial Differential Equation (PDE)
\begin{equation}\label{eq:PDE}
    \dfrac{\partial u(x,t)}{\partial t}-\dfrac{\partial}{\partial x} \left(\mathfrak{D}(x) \ \dfrac{\partial u(x,t)}{\partial x}  \right)=0, \qquad \qquad \begin{matrix}
        0\leq x \leq L, \\
        0\leq t \leq T.
    \end{matrix}
\end{equation}
The PDE \eqref{eq:PDE} serves as a versatile mathematical framework for modeling a broad spectrum of phenomena extending beyond chemical diffusion and heat conduction \cite{Calore}. For instance, in \cite{Joyner}, it is derived from the Black–Scholes model, thereby broadening its applications to qualitative finance. Similarly, in \cite{Piante}, the same equation is utilized to describe nutrient uptake by plant root hairs. Furthermore, as argued in \cite{Bridge}, it may be applicable in simulating the onset of corrosion in concrete bridge beams due to chloride ion diffusion.

Our experiment is conducted within a conservative setting, wherein we define initial conditions and Neumann zero-flux boundary conditions as follows
\begin{equation}\label{eq:PDE_Conditions}
u(x,0)=f(x), \qquad \frac{\partial u}{\partial x}(0,t)=\frac{\partial u}{\partial x}(L,t)=0.
\end{equation}
In this scenario, the conservation law
\begin{equation}\label{eq:PDE_onservazione}
\int_{0}^L u(x,t) \ dx =
\int_{0}^L f(x) \ dx, \qquad \forall t\geq 0,
\end{equation}
holds true since integration with respect to $x$ of both sides of \eqref{eq:PDE} leads to
\begin{equation*}
\frac{d}{dt} \int_{0}^L u(x,t) \ dx = \mathfrak{D}(L) \frac{\partial u(L,t)}{\partial x} - \mathfrak{D}(0) \frac{\partial u(0,t)}{\partial x} = 0.
\end{equation*}
In order to obtain a numerical solution to equation \eqref{eq:PDE}-\eqref{eq:PDE_Conditions} that preserves positivity and satisfies a discrete equivalent of \eqref{eq:PDE_onservazione}, we introduce a finite volume semi-discretization for the spatial variable (see, for instance, \cite[Chapter 4]{LeVeque_2002}) and subsequently employ modified Patankar schemes to integrate the resulting system of ordinary differential equations.
\subsubsection{Spatial semi-discretization and conservative PDS} 
Let $\Delta x>0$ and $\{x_j\}_{j \geq 0}$ be a uniform mesh such that $x_j=\left(j+\frac{1}{2}\right)\Delta x,$ represents the center of a one-dimensional cell with edges $x_{j-\frac{1}{2}}$ and $x_{j+\frac{1}{2}}.$ Consider the approximations $v_j(t)$ of the average solution value over each grid cell, defined as
\begin{equation}\label{eq:significato_approx_volumi}
    v_j(t) \approx u_j(t)=\dfrac{1}{\Delta x} \int_{x_{j-\frac{1}{2}}}^{x_{j+\frac{1}{2}}} u(x,t) \ dx, \qquad \; j=0,\dots, N_x,
\end{equation}
with $\Delta x  N_x =L.$ Integrating the PDE \eqref{eq:PDE} over the $j$-th cell and dividing by $\Delta x$ yields the exact differential rule 
\begin{equation}\label{eq:exact_update_rule}
    u^\prime_j(t)=- \frac{1}{\Delta x} \left(F_{j+\frac{1}{2}}(t)-F_{j-\frac{1}{2}}(t)\right), \qquad \quad j=0,\dots, N_x,
\end{equation}
where $F_{j\pm\frac{1}{2}}(t)= -\mathfrak{D}\!\left(x_{j\pm\frac{1}{2}}\right)\partial_x u\!\left(x_{j\pm\frac{1}{2}},t\right)$ denotes the flux through the edges of the cell.
A semi-discrete scheme is then derived by a midpoint approximation of the fluxes in \eqref{eq:exact_update_rule}. Specifically, for the interior cells ($j=1,\dots, N_x-1$), we set
\begin{equation}\label{eq:SemiDiscr_Interni}
    v^\prime_j(t)=\dfrac{\mathfrak{D}\!\left(x_{j+\frac{1}{2}}\right)v_{j+1}(t)-\left(\mathfrak{D}\!\left(x_{j+\frac{1}{2}}\right)-\mathfrak{D}\!\left(x_{j-\frac{1}{2}}\right)\right)v_{j}(t)+\mathfrak{D}\!\left(x_{j-\frac{1}{2}}\right)v_{j-1}(t)}{\Delta x^2},
\end{equation}
while for the boundary ones the conditions \eqref{eq:PDE_Conditions} become
\begin{equation}\label{eq:SemiDiscr_Frontiera}
    v^\prime_0(t)= \mathfrak{D}\!\left(x_{\frac{1}{2}}\right)\dfrac{v_{1}(t)-v_{0}(t)}{\Delta x^2}, \qquad v^\prime_{N_x}(t)= \mathfrak{D}\!\left(x_{N_x-\frac{1}{2}}\right)\dfrac{v_{N_x-1}(t)-v_{N_x}(t)}{\Delta x^2}.
\end{equation}
Here, fictitious cells with centers $x_{-1}$ and $x_{N_x +1},$ located just beyond the computational domain, have been introduced for deriving \eqref{eq:SemiDiscr_Frontiera}.

The second order semi-discretization scheme \eqref{eq:SemiDiscr_Interni}-\eqref{eq:SemiDiscr_Frontiera} corresponds to the linear system of ordinary differential equations 
\begin{equation}\label{eq:PDE_to_ODE}
    \begin{split}
        \bm{v}^\prime(t)&=\dfrac{1}{\Delta x^2}
        \begin{pmatrix}
            - \mathfrak{D}_{\frac{1}{2}} &  \mathfrak{D}_{\frac{1}{2}} & & \\
            \mathfrak{D}_{\frac{1}{2}} & -\mathfrak{D}_{\frac{1}{2}}-\mathfrak{D}_{\frac{3}{2}} & \mathfrak{D}_{\frac{3}{2}} & \\
             & \ddots & \ddots & \ddots &  \\
             &  \mathfrak{D}_{N_x-\frac{3}{2}} & -\mathfrak{D}_{N_x-\frac{3}{2}}-\mathfrak{D}_{N_x-\frac{1}{2}} & \mathfrak{D}_{N_x-\frac{1}{2}} \\
             & & \mathfrak{D}_{N_x-\frac{1}{2}} & -\mathfrak{D}_{N_x-\frac{1}{2}}
             \end{pmatrix} 
        \cdot\bm{v}(t) \\
        &=A(\Delta x) \cdot\bm{v}(t),
    \end{split}
\end{equation}
with $\bm{v}(t)=(v_0(t),\dots,v_{N_x}(t))^\mathsf{T}\in \mathbb{R}^{N_x+1}$ and $\mathfrak{D}_j=\mathfrak{D}(x_j),$ for each $j.$ The following result proves, for the solution to \eqref{eq:PDE_to_ODE}, the semi-discrete counterpart of the conservation law \eqref{eq:PDE_onservazione}. 
\begin{theorem}\label{thm:Spatial_Conservative}
		Let $u(x,t)$ be the continuous solution to \eqref{eq:PDE}-\eqref{eq:PDE_Conditions} for $(x,t)\in [0,L]\times[0,T],$ with positive $L$ and $T.$ Let $\{v_j(t)\}_{j \geq 0}$ be its approximation (in the sense of \eqref{eq:significato_approx_volumi}) computed by \eqref{eq:PDE_to_ODE} with $\Delta x=L/N_x.$ Then, independently of $\Delta x>0,$
        \begin{equation}\label{eq:PDE_onservazione_semidisc}
            \Delta x \sum_{j=0}^{N_x} v_j(t)=\Delta x \sum_{j=0}^{N_x} f(x_j), \qquad \qquad \forall t\geq 0.
        \end{equation}
\end{theorem}
\begin{proof}
    Given $\bm{f}=(f(x_0),\dots,f(x_{N_x}))^\mathsf{T},$ the equality $\Delta x \ \bm{e}^\mathsf{T} \bm{v}(0)=\Delta x \ \bm{e}^\mathsf{T}\bm{f}$ directly comes from the boundary conditions \eqref{eq:PDE_Conditions}. Therefore, for assuring \eqref{eq:PDE_onservazione_semidisc}, it suffices to show that $(\bm{e}^\mathsf{T} v_j(t))^\prime=0.$ The matrix $A\in \mathbb{R}^{(N_x+1)\times (N_x+1)} $ is symmetric, hence  $\bm{e}^\mathsf{T} \bm{v}^\prime(t)=\bm{e}^\mathsf{T} (A \bm{v}(t))=(A \bm{e})^\mathsf{T}\bm{v}(t)=\bm{0}^\mathsf{T}\bm{v}(t)=0,$ which yields the result.
\end{proof}

The differential system \eqref{eq:PDE_to_ODE} can be rewritten as a PDS of the form \eqref{eq:sistema_compatto} with $\bm{y}(t)=\bm{v}(t)\in \mathbb{R}^{N_x+1},$ $D(\bm{y}(t))=P(\bm{y}(t))^\mathsf{T}\in \mathbb{R}^{(N_x+1)\times (N_x+1)}$ and
\begin{equation}\label{eq:PDE_PDS}
    P(\bm{y})=\dfrac{1}{\Delta x^2}
        \begin{pmatrix}
            0 &  y_2 \mathfrak{D}_{\frac{1}{2}}  & & \\
            y_1 \mathfrak{D}_{\frac{1}{2}} & 0 & y_3 \mathfrak{D}_{\frac{3}{2}} & \\
             & \ddots & \ddots & \ddots &  \\
             & y_{N_x-2} \mathfrak{D}_{N_x-\frac{3}{2}} & 0 & y_{N_x} \mathfrak{D}_{N_x-\frac{1}{2}} \\
             & & y_{N_x-1} \mathfrak{D}_{N_x-\frac{1}{2}} & 0
             \end{pmatrix}.
\end{equation}
Thus, once the equivalence of \eqref{eq:PDE_to_ODE} with a fully conservative PDS is established, the property \eqref{thm:Convergence} proved with Theorem \ref{thm:Spatial_Conservative} automatically follows from \eqref{eq:int_primo}.

\subsubsection{Simulation results}
For our numerical experiments, we consider the PDE \eqref{eq:PDE}-\eqref{eq:PDE_Conditions} with $L=1,$ $T=60$ and initial condition
\begin{equation*}
    f(x)=2-2\sin^2\left(\dfrac{\pi}{2}-\dfrac{1}{4}\right).
\end{equation*}
Moreover, a space-dependent diffusion coefficient
\begin{equation*}
    \mathfrak{D}(x)=D_0 \left(x-\frac{2}{3}\right)^2 \frac{\tan^{-1}(2x-3)}{(2x-3)}+10^{-5}, \qquad \quad \mbox{with} \; D_0=10^{-2},
\end{equation*}
is introduced to simulate heterogeneous diffusion phenomena (see, for instance, \cite{Molly} and references therein).

Here, the solution to the PDE is approximated by integrating the PDS \eqref{eq:sistema_compatto}-\eqref{eq:PDE_PDS} with modified Patankar methods. Since the stability investigation of the semi-discretization \eqref{eq:PDE_to_ODE} falls outside the scope of this work, we empirically adjust the temporal step size $h$ to be smaller than the spatial one $\Delta x$. The outcomes of the simulation by the MPLM-$7(5)$ scheme with $h=5\cdot 10^{-4}$ and $\Delta x=5\cdot 10^{-3},$ are presented in Figure \ref{fig:Diffusion}. 

Table \ref{tab:Diffusion} presents the mean errors, as defined in \eqref{eq:Relative_Error_all_components}, along with the execution times for the Julia implementations of both the MPLM and MPDeC methods. The numerical residual
\begin{equation*}
   r(h)=\Delta x \max_{0\leq n\leq \frac{T}{h}} \left(\sum_{j=0}^{N_x} \left|v_j(t_n)-f(x_j)\right|\right),
\end{equation*}
on the semi-discrete conservation law \eqref{thm:Spatial_Conservative} is there reported, as well. The experimental results reveal that the MPLM methods demonstrate superior performances (cf. Figure \ref{fig:WPD_Diffusion}) and exhibit a more regular behavior for the residual $r(h).$

\begin{table} \center \footnotesize
% table caption is above the table
\caption{Mean Errors, CPU times and numerical residuals for the different methods applied to \hyperref[Test5]{Test 5} with $\Delta x=10^{-2}$.}
\smallskip
\label{tab:Diffusion}       % Give a unique label
% For LaTeX tables use
\begin{tabular}{cccc|cccc}
\hline\noalign{\smallskip}
\multicolumn{4}{c|}{MPLM-$4(3)$} & \multicolumn{4}{|c}{MPDeC$(3)$}   \\
\noalign{\smallskip}\hline\noalign{\smallskip}
$h$ & $\varepsilon(h)$ & CPU time & $r(h)$ & $h$ & $\varepsilon(h)$ & CPU time & $r(h)$ \\
$2^{-9}$ & $2.73 \cdot 10^{-8}$ & $9.71 \cdot 10^{2}$ & $1.42 \cdot 10^{-15}$ & $2^{-9}$ & $7.97 \cdot 10^{-9}$ & $2.85 \cdot 10^{3}$ & $1.47 \cdot 10^{-13}$ \\
$2^{-10}$ & $4.71 \cdot 10^{-9}$ & $2.03 \cdot 10^{3}$ & $4.55 \cdot 10^{-15}$ & $2^{-10}$ & $1.12 \cdot 10^{-9}$ & $5.74 \cdot 10^{3}$ & $2.84 \cdot 10^{-16}$ \\
$2^{-11}$ & $6.95 \cdot 10^{-10}$ & $4.05 \cdot 10^{3}$ & $3.84 \cdot 10^{-14}$ & $2^{-11}$ & $1.51 \cdot 10^{-10}$ & $1.24 \cdot 10^{4}$ & $2.85 \cdot 10^{-16}$ \\
$2^{-12}$ & $9.64 \cdot 10^{-11}$ & $8.43 \cdot 10^{3}$ & $6.11 \cdot 10^{-14}$ & $2^{-12}$ & $3.90 \cdot 10^{-11}$ & $2.77 \cdot 10^{4}$ & $1.56 \cdot 10^{-12}$ \\
\noalign{\smallskip} \hline
\noalign{\smallskip}
\multicolumn{4}{c}{MPLM-$5(4)$} & \multicolumn{4}{|c}{MPDeC$(4)$}   \\
\noalign{\smallskip}\hline\noalign{\smallskip}
$h$ & $\varepsilon(h)$ & CPU time & $r(h)$ & $h$ & $\varepsilon(h)$ & CPU time & $r(h)$ \\
$2^{-9}$ & $4.53 \cdot 10^{-9}$ & $1.42 \cdot 10^{3}$ & $1.85 \cdot 10^{-15}$ & $2^{-9}$ & $5.33 \cdot 10^{-10}$ & $6.27 \cdot 10^{3}$ & $1.42 \cdot 10^{-16}$  \\
$2^{-10}$ & $4.32 \cdot 10^{-10}$ & $2.86 \cdot 10^{3}$ & $3.41 \cdot 10^{-15}$ & $2^{-10}$ & $4.35 \cdot 10^{-11}$ & $1.31 \cdot 10^{4}$  & $1.42 \cdot 10^{-16}$  \\
$2^{-11}$ & $3.89 \cdot 10^{-11}$ & $5.66 \cdot 10^{3}$ & $3.55 \cdot 10^{-15}$ & $2^{-11}$ & $3.58 \cdot 10^{-11}$ & $2.60 \cdot 10^{4}$  & $6.11\cdot 10^{-12}$ \\
$2^{-12}$ & $3.57 \cdot 10^{-11}$ & $1.14 \cdot 10^{4}$ & $2.50 \cdot 10^{-14}$ & $2^{-12}$ & $3.61 \cdot 10^{-11}$ & $5.46 \cdot 10^{4}$  & $5.30 \cdot 10^{-12}$ \\
\noalign{\smallskip} \hline
\end{tabular}
\end{table}

\begin{figure}
\begin{center}
	\scalebox{1}{\includegraphics{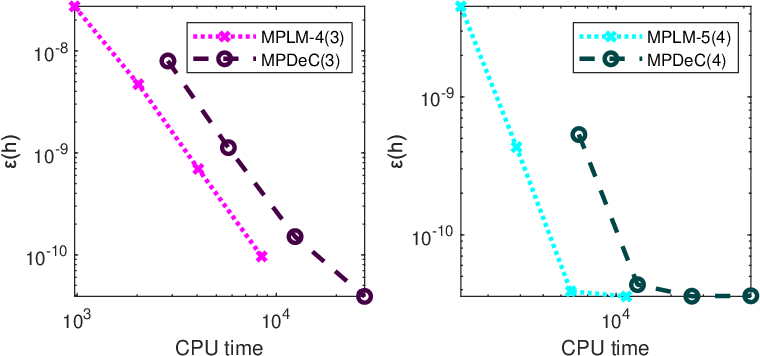}}
\end{center}%\vspace{-0.65cm}
\caption{Work Precision Diagram: Mean Error versus CPU time for the different methods applied to \hyperref[Test5]{Test 5} with $\Delta x=10^{-2}$ and $h=2^{-m}$, $m=9,\ldots,12.$}
\label{fig:WPD_Diffusion}
\end{figure}
% CONTROLLARE: dovrebbe essere $\Delta x=10^{-2}$?
% GRAZIE MILLE PROF IZZO, HO CORRETTO.

\begin{figure}
\begin{center}
	\scalebox{1}{\includegraphics{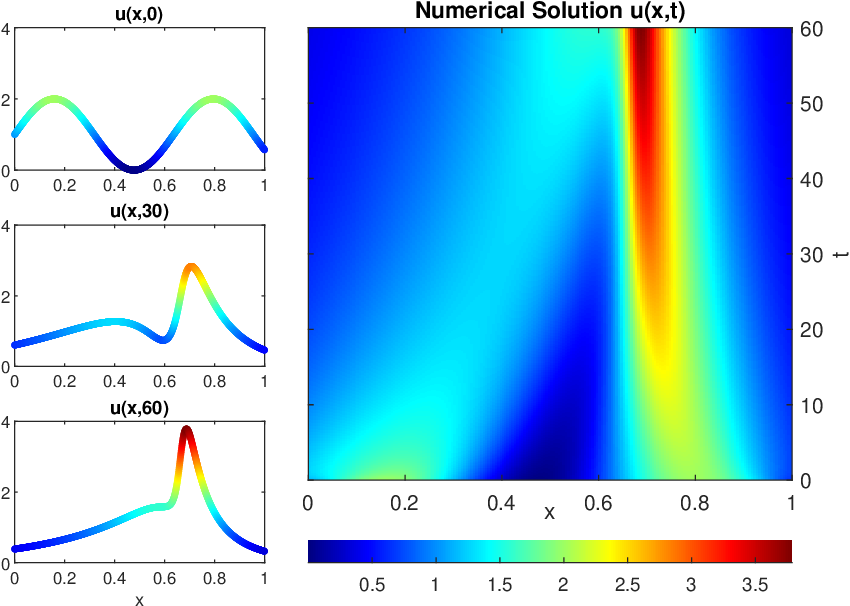}}
\end{center}%\vspace{-0.65cm}
\caption{Numerical solution of \hyperref[Test5]{Test 5} by MPLM-$7(5)$ with $N_x=2000$ and $h=5\cdot 10^{-4}.$}
\label{fig:Diffusion}
\end{figure}

\section{Conclusions and perspectives}\label{sec:Conclusions}
In this manuscript, we introduced accurate linearly implicit time integrators specifically designed for production-destruction differential systems. Notably, the extension of the modified Patankar technique to multistep schemes has resulted in conservative numerical methods which retain, with no restrictions on the discretization steplength, the positivity of the solution and the linear invariant of the system. We carried out a theoretical investigation of the properties of the Patankar weight denominators which ensure the consistency and convergence of the proposed schemes. Additionally, we devised an embedding technique to practically compute the PWDs and achieve arbitrarily high order of convergence. The numerical tests conducted on various problems provided experimental confirmation of the theoretical findings. The comparison with the third-order modified Patankar Runge-Kutta method presented in \cite[Lemma 6, Case II with $\gamma=0.5$]{Kopecz2018} highlighted the superior performance of the proposed MPLM-$k(p)$ integrators. Furthermore, the MPLM-$k(p)$ schemes proved to be competitive with the high order modified Patankar deferred correction discretizations in \cite{Torlo2020}, especially in the case of high-dimensional systems and  vanishing initial states.

Given the results of this paper, MPLM-$k(p)$ methods demonstrate considerable potential for the numerical integration of production-destruction systems.  However, several aspects require deeper analysis, which we plan to address in future work. First of all, the optimization of coefficients selection for the underlying linear multistep method may be investigated, considering various combinations of $k$ and $p$ to ensure both positivity constraint and high order of convergence. The possibility of an extension to include negative coefficients may be considered, as well. Furthermore, since the methods of Table \ref{tab:1} exhibit reduced efficiency on stiff ODEs tests such as the Robertson problem \cite{Kopecz2018second,Huang2019II}, the necessity arises of a comprehensive MPLM stability analysis. In this regard, the implementation of variable stepsize approaches and local error control strategies may reveal of high interest.

%\vspace{1cm}
%\noindent {\bf Acknowledgment} 
%\\This work was supported by GNCS-INdAM.

% \section{Acknowledgment}
% This work was supported by the Italian MUR under the PRIN 2022 project No. 2022N3ZNAX and the PRIN 2022 PNRR project No. P2022WC2ZZ, and by the INdAM under the GNCS Project E53C23001670001.

% {\em Oppure specifichiamo per ogni autore:}

% The work of E. Messina was funded partially by the Italian MUR under the PRIN 2022 PNRR project No. P2022WC2ZZ and partially by the INdAM under the GNCS Project E53C23001670001.
% The work of G.Izzo was funded partially by the Italian MUR under the PRIN 2022 project No. 2022N3ZNAX and the PRIN 2022 PNRR project No. P2022WC2ZZ, and partially by the INdAM under the GNCS Project E53C23001670001.
% The work of M. Pezzella was funded by the INdAM under the GNCS Project E53C23001670001.

\section*{Appendix A. Concistency Necessary Condition}\label{Appendix}
In section \ref{sec:Convergence} we introduced, with Theorem \ref{thm:Consis_order_p}, a condition on the Patankar weight denominators which leads to order $p$ consistent MPLM schemes. Here, our objective is to prove Theorem \ref{Them_Con_Neces} showing that \eqref{eq:cond_sigma} represents a necessary requirement for the consistency, as well. To do that, we consider the particular production-destruction system \eqref{eq:PDS} with
\begin{equation}\label{eq: particular_PDS}
	\begin{split}
		P_{i}(\bm{y})=\begin{cases}
			\mu y_i, \;\; &\mbox{if}\; i=j^*,\phantom{j=i^*,} \\
			0, &\mbox{otherwise,}
		\end{cases} \!\! \mbox{and} \;\;\;\;\; D_{i}(\bm{y})=\begin{cases}
			\mu y_i, \;\; &\mbox{if}\; i=i^*,\phantom{j=j^*,} \\
			0, &\mbox{otherwise,}
		\end{cases}
		%\\
		%p_{ij}(\bm{y})=\begin{cases}
		%	\mu y_i, \;\; &\mbox{if}\; i=j^*,j=i^*, \\
		%	0, &\mbox{otherwise,}
		%\end{cases} \;\; \mbox{and} \;\; d_{ij}(\bm{y})=\begin{cases}
		%	\mu y_i, \;\; &\mbox{if}\; i=i^*,j=j^*, \\
		%	0, &\mbox{otherwise,}
		%\end{cases}
	\end{split}
\end{equation}
where $i^*$ and $j^*$ are fixed indices in $\{1,\dots,N\}$ and $\mu$ is a given positive constant. It can be proved that 
\begin{equation*}
	y_{i^*}(t)=\exp(-\mu t),\qquad y_{j^*}(t)=2-\exp(-\mu t),\qquad y_{l}(t)=1, \qquad \begin{array}{rr}
	     1\leq l\leq N,  \\
	     l\neq i^*, \;
      l \neq j^*, 
	\end{array}
\end{equation*}
is the unique solution of the positive and fully conservative PDS \eqref{eq:PDS}-\eqref{eq: particular_PDS}. The following investigation outlines the behaviour of the MPLM discretizations of the class \eqref{eq:Compact_Expression} applied to \eqref{eq:PDS}-\eqref{eq: particular_PDS}.  % and points out necessary conditions for the order $p$ consistency of the MPLM-$k$ method. 

\begin{customproof}{Them_Con_Neces}
    The order $p$ convergence of the underlying LM discretization and the consistency hypothesis on the MPLM-$k$ method \eqref{eq:Compact_Expression}, imply
    \begin{equation*}
        \delta_i^{LM}(h;t_n)=\mathcal{O}(h^{p+1}) \qquad \mbox{and} \qquad \delta_i(h;t_n)=\mathcal{O}(h^{p+1}), \qquad
            n\geq k, \quad i=1,\dots,N,
    \end{equation*}
    where the local errors $\delta_i^{LM}(h;t_n)$ and $\delta_i(h;t_n)$ are defined in \eqref{eq:delta_LM} and \eqref{eq:Local_Error}, respectively. Subtracting the former from the latter yields
    \begin{equation*}
        \begin{split}
            &\sum_{r=1}^{k}\beta_r\sum_{j=1}^{N}\left(p_{ij}(\bm{y}(t_{n-r}))\left(1-\dfrac{y_j(t_n)}{\sigma_j(\bm{y}(t_{n-1}),\dots,\bm{y}(t_{n-k}))}\right)\right)\\
            -&\sum_{r=1}^{k}\beta_r\sum_{j=1}^{N}\left(d_{ij}(\bm{y}(t_{n-r}))\left(1-\dfrac{y_i(t_n)}{\sigma_i(\bm{y}(t_{n-1}),\dots,\bm{y}(t_{n-k}))}\right)\right)=\mathcal{O}(h^{p}),
        \end{split}
    \end{equation*}
    for $i=1,\dots,N$ and $n\geq k.$  Furthermore, for the particular PDS \eqref{eq: particular_PDS}, taken $i=i^*,$ 
    \begin{equation*}
        \mu\sum_{r=1}^{k}\beta_r y_{i^*}(t_{n-r})\left(1-\dfrac{y_{i^*}(t_n)}{\sigma_{i^*}(\bm{y}(t_{n-1}),\dots,\bm{y}(t_{n-k}))}\right)=\mathcal{O}(h^{p}).
    \end{equation*}
    Therefore, the result comes from the positivity of the system and the arbitrariness of the choice of $1\leq i^*\leq N.$
\end{customproof}

\section*{Acknowledgment}
This work was supported by the Italian MUR under the PRIN 2022 project No. 2022N3ZNAX and the PRIN 2022 PNRR project No. P2022WC2ZZ, and by the INdAM under the GNCS Project E53C23001670001.

% \printbibliography % Stampa la bibliografia

%\bibliographystyle{plain}
%\bibliography{references}  %%% Uncomment this line and comment out the ``thebibliography'' section below to use the external .bib file (using bibtex) .

%%% Uncomment this section and comment out the \bibliography{references} line above to use inline references.

\end{document}